\documentclass[12pt,twoside,a4paper]{amsart}
\usepackage{amssymb}
\usepackage{latexsym}
\usepackage{amsmath}
\usepackage{euscript}
\usepackage{graphics}

\newtheorem{theorem}{Theorem}[section]
\newtheorem{corollary}[theorem]{Corollary}
\newtheorem{lemma}[theorem]{Lemma}
\newtheorem{proposition}[theorem]{Proposition}
\theoremstyle{definition}
\newtheorem{definition}[theorem]{Definition}
\theoremstyle{remark}
\newtheorem{remark}[theorem]{Remark}
\numberwithin{equation}{section}
\renewenvironment {proof} {\begin{trivlist} \item[\hspace{\labelsep}%
\sc Proof.]}{$\Box$ \end{trivlist}}

      \def\dC{{\mathbb C}}

   \def\dN{{\mathbb N}}   
      \def\dR{{\mathbb R}}

   \def\dZ{{\mathbb Z}}

   \def\cH{{\mathcal H}}   
      
\def\cM{{\mathcal M}}   \def\cN{{\mathcal N}}   
      
      \def\cU{{\mathcal U}}
      
   \def\cZ{{\mathcal Z}}

\font\gothic=eufm10
\def\gh{\mbox{\gothic\char'150}}

\newcommand {\wt}{\widetilde}
\newcommand {\wh}{\widehat}

\usepackage[usenames,dvipsnames,svgnames,table]{xcolor}

\begin{document}

\title[Full indefinite Stieltjes moment problem]
{Full indefinite Stieltjes moment problem and Pad\'{e} approximants }

\author[Volodymyr Derkach]{Volodymyr Derkach}
\address{Department of Mathematics\\ Vasyl$'$ Stus Donetsk National University\\ 600-Richchya Str 21, Vinnytsia, 21021,  Ukraine}
\email{derkach.v@gmail.com}

\author[Ivan Kovalyov]{Ivan Kovalyov}
\address{Department of Mathematics \\Dragomanov National Pedagogical University \\Kiev, Pirogova 9, 01601, Ukraine}
\email{i.m.kovalyov@gmail.com}

\thanks{The present research was supported
by Ministry of Education and Science of Ukraine (projects \# 0118U003138, 0118U002060), and a grant of the Volkswagen Foundation. V. D. gratefully
acknowledges financial support by the German Research Foundation (DFG), grant TR 903/22-1}



 \begin{abstract}
 Full indefinite Stieltjes moment problem is studied via the step-by-step Schur algorithm.
Naturally associated with indefinite Stieltjes moment problem are generalized Stieltjes  continued fraction and a system of difference equations, which, in turn, lead to factorization of resolvent matrices of indefinite Stieltjes moment problem.
A criterion for such a problem to be indeterminate in terms of continued fraction is found and a complete description of its  solutions is given in the indeterminate case.
Explicit formulae for diagonal and sub-diagonal Pad\'{e} approximants for formal power series corresponding to indefinite Stieltjes moment problem and   convergence results for Pad\'{e} approximants are presented.
\end{abstract}

\maketitle

\section{Introduction}
The classical Stieltjes moment problem 
consists in the following:
given a sequence of real numbers $s_{j}$
$(j\in\mathbb{Z}_{+}:=\mathbb{N}\cup\{0\})$ find a  positive
measure $\sigma$ with a support on $\mathbb{R}_{+}$, such that
\begin{equation}\label{int1}
    \int_{\mathbb{R}_{+}}t^jd\sigma(t)=s_{j},\quad j\in\mathbb{Z}_{+}.
\end{equation}

It follows easily from~\eqref{int1} that the inequalities
\begin{equation}\label{int5}
    S_{n}:=\left(s_{i+j}\right)_{i,j=0}^{n-1}\geq0\quad\mbox{ and }\quad S_{n}^{+}:=\left(s_{i+j+1}\right)_{i,j=0}^{n-1}\geq0, \quad n\in\mathbb{Z}_{+},
\end{equation}
are necessary for solvability of the moment problem~\eqref{int1}. Moreover, the inequalities~\eqref{int5} 
are also sufficient for solvability of the moment problem~\eqref{int1}, see~\cite[Appendix, Section A1]{Akh}.
Let
\begin{equation}\label{int5xy}
   D_n:=\mbox{det } S_{n}\quad\mbox{ and }\quad  \quad D_n^+:=\mbox{det } S_{n}^{+}, \quad n\in\mathbb{Z}_{+}.
\end{equation}

{
For every solution $\sigma$ of the Stieltjes moment problem its Stieltjes transform
\begin{equation}\label{int3}
    f(z)=\int_{\mathbb{R}_{+}}\frac{d\sigma(t)}{t-z}\qquad z\in \mathbb{C}\backslash \mathbb{R}_{+}
\end{equation}
belongs to the class
 $\mathbf{N}$ of functions holomorphic on $\mathbb{C}\backslash \mathbb{R}$ with nonnegative imaginary part in  $\mathbb{C}_{+}$ and such that $f(\overline{z})=\overline{f(z)}$ for $z\in \mathbb{C}_{+}$.
Moreover, $f$ belongs to the {\it Stieltjes class} $\mathbf{S}$ of
functions $f\in \mathbf{N}$, which admit holomorphic and nonnegative
continuation to $\mathbb{R}_{-}$. By M.G. Krein's
criterion~\cite{KacK68}
\[
    f\in \mathbf{S}\Longleftrightarrow f\in \mathbf{N} \quad\mbox{and}\quad zf\in \mathbf{N}.
\]

Notice, that by the Hamburger--Nevanlinna theorem~\cite[Theorem 3.2.1]{Akh}
the  Stieltjes moment problem can be reformulated as the following interpolation problem at $\infty$ for the Stieltjes transform $f(z)=\int_{\mathbb{R}_{+}}\frac{d\sigma(t)}{t-z}$ of the measure $\sigma$
\begin{equation}\label{int4}
    f(z)=-\frac{s_{0}}{z}-\frac{s_{1}}{z^2}-\cdots-\frac{s_{2n}}{z^{2n+1}}
    +o\left(\frac{1}{z^{2n+1}}\right),\quad\quad z\widehat{\rightarrow}\infty.
\end{equation}
The notation $z\widehat{\rightarrow}\infty$ means that $z\rightarrow\infty$  nontangentially, that is inside the
sector $\varepsilon<\arg z<\pi-\varepsilon$ for some
$\varepsilon>0$.

In the pioneering paper~\cite{St94} by T. Stieltjes it was shown that if $D_n,$ $D_n^+$ are positive for all $n\in\dZ_+$ and
\begin{equation}\label{int5x}
   l_n:=\frac{D_n^2}{D_n^+D_{n-1}^+},\quad  m_n:=\frac{(D_{n-1}^+)^2}{D_nD_{n-1}},\quad n\in\mathbb{N}.
\end{equation}
then the Laurent series $-\sum_{j=0}^\infty\frac{s_j}{z^{j+1}}$ can be uniquely developed in   the continued fraction
\begin{equation}\label{s4.2Cx}
    \displaystyle\frac{\displaystyle1}{\displaystyle -z m_{1}+\displaystyle\frac{1}{\displaystyle
    l_{1}+\dots\frac{1}{\displaystyle -z m_{n} +\frac{1}{\displaystyle
    l_{n}+\dots}}}}. 
\end{equation}
The moment problem~\eqref{int1} is called {\it determinate}, if it has unique solution, and {\it indeterminate} otherwise.
A solution $\sigma$ of~\eqref{int1} is called canonical, if the set of polynomials is dense in $L^2(\sigma)$.
As was shown in~\cite{St94} the moment problem~\eqref{int1} is indeterminate, if and only if
\begin{equation}\label{eq:Undet}
  M:=\sum_{i=1}^\infty m_i<\infty\quad\mbox{and}\quad L:=\sum_{i=1}^\infty l_i<\infty.
\end{equation}


As was shown in \cite[Appendix, $13.2^\circ$]{KacK68}, in the indeterminate case there exists an entire matrix valued function  $W(z)=(w_{i,j}(z))_{i,j=1}^2$ such that the set of all solutions of the problem~\eqref{int1} can be parametrized by the formula
\begin{equation}\label{eq:StMP}
  \int_{\mathbb{R}_{+}}\frac{d\sigma(t)}{t-z}=T_{W(z)}[\tau(z)]:=
  \frac{w_{11}(z)\tau(z)+w_{12}(z)}{w_{21}(z)\tau(z)+w_{22}(z)},
\end{equation}
where $\tau$ ranges over the Stieltjes class $\mathbf{S}$.
A subset of its canonical solutions was described already in~\cite{St94}.
}

Now let us remind the indefinite versions of the classes  $\mathbf{N}$ and  $\mathbf{S}$, see~\cite{KL77}.
\begin{definition}\label{def:Nk}({\cite{KL77}})
    A function $f$ meromorphic on $\mathbb{C}\backslash\mathbb{R}$ with the set of holomorphy ${\mathfrak h}_f$
    is said to be in the generalized Nevanlinna class $\mathbf{N}_{\kappa}$ $(\kappa\in\mathbb{N})$, if for every set $z_{j}\in \mathbb{C}_{+}\cap{\mathfrak h}_f$ ($z_{i}\neq\overline{z}_{j}$, $i,j=1,\ldots,n$) the form
\[
    \sum_{i,j=1}^{n}\frac{f(z_{i})-\overline{f(z_{j})}}{z_{i}-\overline{z}_{j}}
    \xi_{i}\overline{\xi}_{j} 
\]
has at most $\kappa$ and for some choice of $z_{j}$ ($j=1,\ldots,n$) exactly $\kappa$ negative squares.
 A functions $f\in \mathbf{N}_{\kappa}$ is said to belong to the generalized  Stieltjes class $\mathbf{N}_{\kappa}^{+}$, if $zf\in \mathbf{N}$.
\end{definition}
 Similarly, in \cite{D91,D97} the class $\mathbf{N}_{\kappa}^{k}$  ($\kappa, k\in \mathbb{N}$) was introduced as the set of functions $f\in \mathbf{N}_{\kappa}$, such that $zf(z)$ belongs to the class  $\mathbf{N}_{k}$, see also~\cite{DM87, DM97}, where the class $\mathbf{N}_{0}^{k}$ was studied.

In the present paper we consider the following problems. 

\noindent{\bf Full indefinite moment problem} $MP_{\kappa}(\textbf{s})$.
Given $\kappa\in\mathbb{Z}_{+}$,  and an infinite sequence $\textbf{s}=\left\{s_{j}\right\}_{j=0}^{ \infty}$ of real numbers, describe the set $\mathcal{M}_{\kappa}({\mathbf s})$  of functions $f\in \mathbf{N}_{\kappa}$, which satisfy~\eqref{3p.2.3x} for all $\ell\in\dN$.

\noindent{\bf Full indefinite moment problem} $MP_{\kappa}^{k}(\textbf{s})$.
Given $\kappa,k\in\mathbb{Z}_{+}$,  and an infinite sequence $\textbf{s}=\left\{s_{j}\right\}_{j=0}^{ \infty}$of real numbers, describe the set $\mathcal{M}_{\kappa}^{k}({\mathbf s}):=\mathcal{M}_{\kappa}({\mathbf s})\cap \mathbf{N}_{\kappa}^{k}$.

Indefinite moment problems $MP_{\kappa}(\textbf{s})$ and $MP_{\kappa}^{+}(\textbf{s})$  were studied in \cite{KL79}, \cite{KL81} by the methods of extension theory of Pontryagin space symmetric operators developed in \cite{KL77}, \cite{KL79}.  In particular, it
was shown in \cite{KL79} that the moment problem {$MP^+_{\kappa}(\textbf{s})$  is solvable
if and only if } the number $\nu_{-}(S_{n})$ of
negative eigenvalues of $S_{n}$ does not exceed $\kappa$
and $S_{n}^{+}>0$ for all $n\in \mathbb{N}$. Further applications of the operator approach to the moment problem $MP_{\kappa}^{k}(\textbf{s})$ were given in~\cite{D97}. A reproducing kernel approach to the moment problems $MP_{\kappa}(\textbf{s})$ was presented in~\cite{Dym89}.
A step-by-step algorithm of solving the moment problems $MP_{\kappa}(\textbf{s})$ was
elaborated in \cite{Der03}, \cite{DD04} and \cite{ADL04}.
Applications of the Schur algorithm to degenerate moment problem in the class $\mathbf{N}_{\kappa}$ were given in~\cite{DHS12}.

The step-by-step algorithm reduces the moment problems $MP_{\kappa}(\textbf{s})$ and $MP_{\kappa}^k(\textbf{s})$ to the following truncated  moment problems:

\noindent{\bf Truncated indefinite moment problem} $MP_{\kappa}(\textbf{s},  \ell)$.
Given $\ell,\kappa\in\mathbb{Z}_{+}$,  and a finite sequence $\textbf{s}=\left\{s_{j}\right\}_{j=0}^{ \ell}$ of real numbers, describe the set $\mathcal{M}_{\kappa}({\mathbf s},\ell)$  of functions $f\in \mathbf{N}_{\kappa}$, which satisfy the asymptotic expansion
\begin{equation}\label{3p.2.3x}
    f(z)=-\frac{s_{0}}{z}-\cdots-\frac{s_{\ell}}{z^{\ell+1}}+
o\left(\frac{1}{z^{\ell+1}}\right),\quad
(z=iy, \,\, y{\uparrow}\infty).
\end{equation}

\noindent{\bf Truncated indefinite moment problem} $MP_{\kappa}^{k}(\textbf{s},  \ell)$.
Given $\ell,\kappa,k\in\mathbb{Z}_{+}$,  and a sequence $\textbf{s}=\left\{s_{j}\right\}_{j=0}^{ \ell}$ of real numbers, describe the set $\mathcal{M}_{\kappa}^{k}({\mathbf s},\ell)$  of functions $f\in \mathbf{N}_{\kappa}^{k}$, which satisfy~\eqref{3p.2.3x}.
A truncated moment problem is called {\it even} or {\it odd} regarding to the oddness of the number $\ell+1$ of given moments.

Let  ${\mathcal{H}}$ be the set of all infinite real sequences
${\bf s}=\{s_j\}_{j=0}^{\infty}$ and let
 ${\mathcal{H}}_{\kappa}$ be the set of sequences
${\bf s}\in{\mathcal{H}}$, such that
\begin{equation}\label{eq:Hkl}
   \nu_{-}(S_n)=\kappa \quad\mbox{for all $n$ big enough}
\end{equation}
Denote by ${\mathcal{H}}_{\kappa}^k$ the set of real sequences
${\bf s}\in{\mathcal{H}}_{\kappa}$, such that
$\{s_{j+1}\}_{j=0}^{\infty}\in{\mathcal{H}}_{k}$, i.e.
\begin{equation}\label{eq:Hkkl}
   \nu_{-}(S_{n}^+)=k \quad\mbox{for all $n$ big enough}.
\end{equation}

A number
$n_{j}\in \mathbb{N}$ is said to be a {\it normal index} of the sequence
${\bf s}\in\cH$, if det$\,S_{n_{j}}\neq0$. The ordered set of normal
indices
\[
n_1<n_2<\dots<n_N
\]
 of the sequence ${\bf s}$ is denoted by $\mathcal{N}({\bf s})$.


As was shown in~\cite{DD04} for every ${\bf s}\in{\mathcal{H}}_{\kappa}$
there exists a sequence { of real numbers $b_{i}\in\dR\backslash\{0\}$, $i\in\mathbb{N}$ and real monic polynomials}
\begin{equation}\label{eq:pol a j2}
a_{i}(z)=z^{\ell_{i}}+a_{\ell_{i}-1}^{(i)}z^{\ell_{i}-1}+\ldots+a_{1}^{(i)}z+a_{0}^{(i)}
\end{equation}
of degree $\ell_{i}=n_{i+1}-n_{i}$, $i\in\mathbb{N}$, such that the convergents of the continued fraction
\begin{equation}\label{eq:Pfrac}
   \frac{ - b_{0}}{\displaystyle {a}_{0}(z)-\frac{b_{1}}{\displaystyle
    {a}_{1}(z)-{\dots-\frac{b_{n}}{a_n(z)-\dots}}}}
\end{equation}
for sufficiently large $n$ have the asymptotic expansion~\eqref{3p.2.3x} for every $\ell\in\dN$.
This fact was known already to L.~Kronecker~\cite{Kro81} and then it was reinvented in~\cite{Der03}.
The pairs $(a_i,b_i)$ are called atoms, see~\cite{Fuhr10} and the continued fraction~\eqref{eq:Pfrac} is called the $P$--fraction,~\cite{Mag62}.

Consider the three-term recurrence relation
\begin{equation}\label{6p.eq:3-termrecrel}
    b_{j}y_{{j-1}}(z)-a_{j}(z)y_{{j}}(z)+y_{{j+1}}(z)=0,
\end{equation}
associated with the sequence of atoms $\{a_i,b_i\}$,
$i\in\mathbb{N}$, and define polynomials $P_{{j}}(z)$ and
$Q_{{j}}(z)$ as solutions of the system
\eqref{6p.eq:3-termrecrel} subject to the initial conditions
\begin{equation}\label{system1.1}
    P_{-1}(z)\equiv0,\mbox{
    }P_{0}(z)\equiv1,\mbox{ }
    Q_{-1}(z)\equiv-1, \mbox{
    }Q_{0}(z)\equiv0.
\end{equation}
The polynomials $P_{{j}}$ and
$Q_{{j}}$ are called Lanzcos polynomials of the first and second kind.
Moreover,  the  $n$-th convergent of the continued fraction~(\ref{eq:Pfrac}) takes the form  (see~\cite[Section 8.3.7]{Fuhr10}).
\[
f^{[n]}(z)=-\frac{Q_{{n}}(z)}{P_{{n}}(z)},\quad n\in\dN.
\]
As was shown in~\cite{DD04} the set $\mathcal{M}_{\kappa}({\mathbf s},2n_j-2)$ can be described in terms of the Lanzcos polynomials of the first and second kind.

 A sequence ${\bf s}\in\cH_\kappa^k$ is called {\it regular} (see~\cite{DK15}), and is designated as $\textbf{s}\in \cH_\kappa^{k,reg}$, if
\begin{equation}\label{eq:int_p}
   D_{n_j}^+= \det S_{n_j}^+\ne 0 \qquad \mbox{for all}\quad  j\in \dN.
\end{equation}
In~\cite{DK16} it was shown that an even  indefinite Stieltjes moment problems $MP_{\kappa}^{k}(\textbf{s},  2n_N-1)$  for regular sequence ${\bf s}$ is solvable if and only if
\begin{equation}\label{eq:Gen_Solv+x}
    \kappa_N:=n_-(S_{n_N})\le\kappa\quad\mbox{and}\quad
    k_N^+:=n_-(S_{n_N}^+)\le k.
\end{equation}
For this problem one step of the Schur algorithm was split in~\cite{DK16} into two substeps and this leads to the expansion of $f\in\cM_{\kappa}^{k}(\textbf{s},  2n_N-2)$ into a generalized Stieltjes continued fraction
\begin{equation}\label{eq:con_fracx}
    f(z)=\frac{1}{-zm_{1}(z)+\frac{\displaystyle1}{\displaystyle l_{1}+\cdots+\frac{\displaystyle1}{\displaystyle-zm_{N}(z)
   +\frac{\displaystyle1}{\displaystyle l_N+ f_N(z)}
    }}},
\end{equation}
where $m_{j}$ are  polynomials, $l_{j}\in\mathbb{R}\backslash\{0\}$ and $f_N$ is a function from the generalized Stieltjes class ${\mathbf N}_{\kappa-\kappa_N}^{k-k_N}$, such that $f_N(z)=o(1)$ as $z\widehat{\rightarrow}\infty$.
If $f\in\cM_{\kappa}^{k}(\textbf{s})$ then $f_N$
has an induced asymptotic expansion
\begin{equation}\label{int4N}
    f_N(z)=-\frac{s^{(N)}_{0}}{z}-\frac{s^{(N)}_{1}}{z^2}-\cdots-\frac{s^{(N)}_{2n}}{z^{2n+1}}
    +o\left(\frac{1}{z^{2n+1}}\right),\quad\quad z\widehat{\rightarrow}\infty,
\end{equation}
i.e. $f_N$ is a solution of an induced moment problem ${MP}_{\kappa-\kappa_N}^{k-k_N}({\mathbf s^{(N)}})$ generated by the sequence $\textbf{s}^{(N)}=(s^{(N)}_{i})_{i=0}^\infty$. Then $f_N\in\cM_{\kappa}^{k}(\textbf{s}^{(N)})$.

Generalized Stieltjes continued fractions were studied in~\cite{DK15}.
Associated to the continued fraction
\eqref{eq:con_fracx} there is a system of difference equations (see~\cite[Section~1]{Wall})
\begin{equation}\label{s4.3Ix}
    \left\{
    \begin{array}{l}
        y_{2j-1}-y_{2j-3}=-zm_{j}(z)y_{2j-2}{,}\\
        y_{2j}-y_{2j-2}=l_{j}y_{2j-1},\\
    \end{array}\right.\quad j\in\dN.
\end{equation}

Define the generalized Stieltjes polynomials $P_{{j}}^+$ and
$Q_{{j}}^+$ of the first and second kind as solutions of the system
\eqref{s4.3Ix} subject to the initial conditions
\begin{equation}\label{system1.1Ix}
    P_{-1}^+(z)\equiv0,\mbox{
    }P_{0}^+(z)\equiv1,\mbox{ }
    Q_{-1}^+(z)\equiv 1, \mbox{
    }Q_{0}^+(z)\equiv0.
\end{equation}
The formula~\eqref{eq:con_fracx} for the set  of solutions of the truncated Stieltjes moment problem { ${MP}_{\kappa}^{k}({\mathbf s},2n_N-1)$} can be rewritten in terms of the generalized Stieltjes polynomials $P_{{j}}^+$ and
$Q_{{j}}^+$, $j=2N-1,2N$.
\begin{theorem} {\rm (\cite{DK17})}\label{thm:4.2}
Let a sequence ${\mathbf s}=\{s_i\}_{i=0}^{\infty}\in\mathcal{H}_{\kappa}^{k}$ be regular, $\mathcal{N}({\mathbf s})=\{n_j\}_{j=1}^\infty$, and let  $P_{{j}}^+$ and
$Q_{{j}}^+$  be generalized Stieltjes polynomials  of the first and second kind. Then:
\begin{enumerate}
  \item [(i)] A nondegenerate even moment problem
$MP_{\kappa}^{k}(\textbf{s},2n_N-1)$
is solvable, if and only if
\eqref{eq:Gen_Solv+x} holds.
\item [(ii)] $f\in \mathcal{M}_{\kappa}^{k}(\textbf{s},
  2n_{N}-1)$ if and only if $f$ admits the
  representation
\begin{equation}\label{eq:LFT_W}
    f(z)=
    \frac{Q^{+}_{2N-1}(z){f_N}(z)+Q^{+}_{2N}(z)}{P^{+}_{2N-1}(z){f_N}(z)+P^{+}_{2N}(z)},
\end{equation}
   where
${f_N}(z)$ satisfies the conditions
\begin{equation}\label{eq:tau_N+x}
    {f_N} \in
    {\mathbf N}_{\kappa-\kappa_{N}}^{k-k_{N}^+}\quad\mbox{and}\quad
  {f_N}(z)=o(1),\quad
    z\widehat{\rightarrow}\infty.
\end{equation}
\end{enumerate}
\end{theorem}

In what follows for every $2\times 2$ matrix $W=(w_{ij})_{i,j=1}^2$ we associate the linear-fractional transformation
\begin{equation}\label{eq:LFT}
T_{W}[\tau]:=\frac{w_{11}\tau+w_{12}}{w_{21}\tau+w_{22}}.
\end{equation}
Denote by ${W_{2N}(z)}$ the coefficient matrix of the linear fractional transform~\eqref{eq:LFT_W}:
\begin{equation}\label{eq:ResM_2N}
  W_{2N}(z)=\left(
                \begin{array}{cc}
                  Q^{+}_{2N-1}(z) & Q^{+}_{2N}(z) \\
                  P^{+}_{2N-1}(z) & P^{+}_{2N}(z) \\
                \end{array}
              \right).
\end{equation}
 Then the formula \eqref{eq:LFT_W} can be rewritten as
\begin{equation}\label{eq:LFTx}
  f(z)=T_{W_{2N}(z)}[f_N(z)].
\end{equation}
The structure of the continued fraction~\eqref{eq:con_fracx} leads to the following  factorization of
the matrix valued function ${W_{2N}(z)}$:
\[
    W_{2N}(z)=M_{1}(z)L_{1}\ldots M_{N}(z)L_{N},
\]
where the matrices $M_j(z)$ and $L_j$ are defined by
\begin{equation}\label{eq:LMjIx}
M_{j}(z)=\begin{pmatrix}
1 & 0 \\
-zm_{j}(z) & 1 \\
\end{pmatrix}{,}\quad\mbox{and}\quad
L_{j}=\begin{pmatrix}
1 & l_{j} \\
0 & 1 \\
\end{pmatrix}\quad j\in \dN.
\end{equation}

Continued fractions of the form \eqref{eq:con_fracx} with positive and negative masses $m_j$ were studied by Beals, Sattinger
and Szmigielski \cite{BSS00} in connection with the theory of multi-peakon solutions of the Camassa-Holm equation. In \cite{EKost14} Eckhardt and  Kostenko showed that  inverse spectral problem for multi-peakon solutions of the
Camassa-Holm equation is solvable in the class of continued fractions of the form~ \eqref{eq:con_fracx} with polynomials
$m_j(z) = d_jz + m_j$ of formal degree 1. In \cite{EKost18}  the spectral theory of continued fractions \eqref{eq:con_fracx} was treated within a classical Hamburger moment problem associated with this sequence.

The  full indefinite Stieltjes moment problem $MP_{\kappa}^{k}(\textbf{s})$ for regular sequences $\textbf{s}\in \cH_\kappa^{k,reg}$ was considered in~\cite{DK17} via the operator approach. The moment problem $MP_{\kappa}^{k}(\textbf{s})$ was treated in~\cite{DK17} as a problem of extension theory for a symmetric operator generated by this generalized Jacobi matrix.
In that paper we used quite advanced tools: the theory of boundary triples developed in~\cite{GG84,DM87,DM95},   and  the M.G. Kre\u{\i}n theory of resolvent matrices extended in~\cite{D91,D99} to the case of indefinite inner spaces.

In the present paper we are going to use  elementary tools in order to make the presentation available for a wider audience.
The main idea  is to use the factorization formula for
the coefficients matrix $W_{2j}(z)$
\begin{equation}\label{eq:W_reduceI}
  W_{2j}(z)=W_{2N}(z)W^{(N)}_{2(j-N)}(z),\quad j>N,\,\, j,N\in\dN,
\end{equation}
which allows to reduce the  indefinite Stieltjes moment problem $MP_{\kappa}^{k}(\textbf{s}, 2j)$ to some classical Stieltjes moment problem { $MP_{0}^{0}(\textbf{s}^{(N)}, 2(j-N))$} with the resolvent matrix $W^{(N)}_{2(j-N)}(z)$. Then all the known results for the classical Stieltjes moment problem can be translated to the indefinite Stieltjes moment problem.
In particular, in Theorem~\ref{thm:Full_MP} it is shown that the problem $MP_{\kappa}^{k}(\textbf{s})$ is indeterminate if and only if
\begin{equation}\label{eq:indeterm}
  M:=\sum_{j=1}^\infty m_j(0)<\infty\quad\mbox{and}\quad
  L:=\sum_{j=1}^\infty l_j<\infty.
\end{equation}
In the classical case, polynomials $m_j(z)$ are constants and the criterion~\eqref{eq:indeterm} coincides with the well known Stieltjes criterion, see~\cite[Theorem 0.4]{Akh}.

If \eqref{eq:indeterm} is in force, then the matrix valued functions ${W}_{2j}(z)$  converge to an entire matrix valued function ${W}_{\infty}^{+}(z)$ of order $1/2$ and
the linear fractional transformation~\eqref{eq:LFTx} generated by the matrix valued function ${W}_{\infty}(z)$ provides a description of the set
$\mathcal{M}_{\kappa}^{k}(\textbf{s})$.

In Section~\ref{sec:5} Pad\'{e} approximants for formal power series corresponding to an indefinite Stieltjes moment problem are calculated.
As was shown in~\cite{DD07} the diagonal Pad\'{e} approximants for formal power series corresponding to an indefinite
Hamburger moment problem are represented as a ratio of the Lanzcos polynomials of the 2-nd and the 1-st kind,~\cite{DD07}.
In Theorem~\ref{prop:sub_Diag_Pade} we show that the sub-diagonal Pad\'{e} approximants of the corresponding formal power series is a ratio of the generalized Stieltjes polynomials of the 2-nd and the 1-st kind. 
In Theorem~\ref{thm:Pade} convergence of Pad\'{e} approximants is derived from the classical results using the formula~\eqref{eq:W_reduceI}.

In Section~\ref{sec:6} the results are illustrated by an example of indefinite moment problem associated with Laguerre polynomials $L_n(z,\alpha)$ in the non-classical case $\alpha<-1$.


\section{Preliminaries}
\subsection{Generalized Nevanlinna functions}


A function $f\in
\mathbf{N}_{\kappa}$  is said to belong to the class $\mathbf{N}_{\kappa,-\ell}$ $(\kappa,\ell\in\dZ_+:=\dN\cup\{0\})$ if   $f$ admits the asymptotic expansion \eqref{3p.2.3x} for some real numbers $s_0,\dots,s_{\ell}$.
Let us also set
\begin{equation}\label{6p.eq:exp2}
    \mathbf{N}_{\kappa,-\infty}:=\bigcap_{n\geq0}\mathbf{N}_{\kappa,-2n}.
\end{equation}
Every real polynomial
$P(z)=p_{\nu}z^{\nu}+
\ldots+p_{1}z+p_{0}$ of
degree $\nu$ belongs to a class $\mathbf{N}_{\kappa_{-}(P)}$, where
the index $\kappa_{-}(P)$ can be evaluated by
(see~\cite[Lemma~3.5]{KL77})
\begin{equation}\label{3p.kappaP}
\kappa_{-}(P)=\left\{
\begin{array}{cl}
\left[\frac{\nu+1}{2}\right],&        \mbox{ if } p_{\nu}<0; \mbox{ and } \nu  \mbox{ is odd };\\
\left[\frac{\nu}{2}\right],&  \mbox{ otherwise } .
\end{array}
\right.
\end{equation}




Recall,  that a  function $f\in \mathbf{N}_{\kappa}$  is said to be from the generalized Stieltjes
class $\mathbf{N}_{\kappa}^{\pm k}$, 
if $z^{\pm 1}f(z)$ belongs to $ \mathbf{N}_k$    $\left(\kappa,k\in\mathbb{Z}_{+}\right)$.
Let us collect some   properties of generalized Nevanlinna functions, see \cite{KL77}, \cite{D91}.
{
\begin{proposition}{\rm (\cite{KL77})} \label{prop:2.1}
Let 
$\kappa,\kappa_1,k\in\dZ_+$. Then the following statements hold:
\begin{enumerate}
  \item [(i)] $f\in \mathbf{N}_{\kappa}\Longleftrightarrow-\frac{1}{f}\in \mathbf{N}_{\kappa}$;
\item [(ii)] $f\in \mathbf{N}_{\kappa}^k\Longleftrightarrow-\frac{1}{f}\in \mathbf{N}_{\kappa}^{-k}$;
\item [(iii)] $f\in \mathbf{N}_{\kappa}^k\Longleftrightarrow zf(z)\in \mathbf{N}_{k}^{-\kappa}$;
\item [(iv)] if $f\in \mathbf{N}_{\kappa}$, $f_{1}\in \mathbf{N}_{\kappa_{1}}$ then $f+f_{1}\in \mathbf{N}_{\kappa'}$, where
  $\kappa'\leq\kappa+\kappa_{1}$. If, in addition,  $f(iy)=o(y)$ as
$y\rightarrow\infty$ and $f_{1}$ is a polynomial,  then 
$f+f_{1}\in \mathbf{N}_{\kappa+\kappa_{1}}$;
\item[(v)] if a function $f\in \mathbf{N}_{\kappa}$ has an asymptotic
expansion~\eqref{int4} for evrey $n\in\dN$,
then there exists $\kappa'\le\kappa$, such that
$\{s_j\}_{j=0}^{\infty}\in{\mathcal{H}}_{\kappa'}$.
  \end{enumerate}
\end{proposition}
}

The notions of generalized poles 
of
non-positive type of a function ${f}\in {\bf N}_\kappa$  were introduced in~\cite{KL81}.
The following definitions are based on \cite{L86}. A point
$\alpha\in\dR$ is called a \textit{generalized pole} of non-positive type
(GPNT) of the function ${f}\in {\bf N}_\kappa$ with multiplicity
$\kappa_\alpha({f})$ if
\begin{equation}
\label{gpol} -\infty < \lim_{z\widehat{\rightarrow }\alpha}
(z-\alpha)^{2\kappa_{\alpha}+1}{f}(z) \leq 0,\quad 0 <
\lim_{z\widehat{\rightarrow }\alpha}
(z-\alpha)^{2\kappa_{\alpha}-1}{f}(z) \leq \infty.
\end{equation}
Similarly,
the point $\infty$ is called a generalized pole of $f$  of
nonpositive type (GPNT) with multiplicity
$\kappa_\infty(f)$ if
\begin{equation}
\label{infgpol} 0\leq\lim_{z\widehat{\rightarrow }\infty }
\frac{{f}(z)}{z^{2\kappa_{\infty}+1}} < \infty,\quad
-\infty\leq\lim_{z\widehat{\rightarrow }\infty }
\frac{{f}(z)}{z^{2\kappa_{\infty}-1}} < 0.
\end{equation}

The following fundamental result was proved in~\cite[Theorem 3.5]{KL81}.
\begin{proposition}\label{prop:GPNT}
Let $f\in \mathbf{N}_{\kappa}$. Then
the total multiplicity of the poles of $f$ in $\dC_+$ and the generalized poles of negative type of $f$ in $\dR\cup\{\infty\}$ is equal to $\kappa$.
\end{proposition}

\subsection{Regular sequences and step--by--step algorithm}\label{sec:2.2}

In the present paper we will consider so-called regular  sequences $\textbf{s}$ from $\cH_{\kappa}^{k}$ introduced in~\cite{DK15}.
{\begin{definition}
 The sequence $\textbf{s}$ is called regular (and is denoted as $\textbf{s}\in \cH_{\kappa}^{k,reg}$), if one of the
 following equivalent conditions holds
 \begin{enumerate}
  \item [({i})] $P_{j}(0)\neq0$ for every $j\in \dN$;
  \item [({ii})] $D_{{n}_{j}}^{+}:=\det S_{n_j}^+\neq0$ for every $j\in \dN$;
  \item [(iii)] $D_{{n}_{j}-1}^{+}\neq0$ for every $j\in \dN$.
  \end{enumerate}
\end{definition}
}
The following step--by--step algorithm of solving the indefinite Stieltjes moment problem for regular sequences ${\bf s}\in \mathcal{H}_{\kappa}^{k,reg}$ was developed in~\cite{DK17}. For { an} indefinite Hamburger moment problem such an algorithm was elaborated in~\cite{Der03,DD04}.

Let $f\in\cM_\kappa^k({\bf s})$ and let $n_1$ be the first normal index, i.e.
\[
s_0=\dots=s_{n_1-2}=0, \quad s_{n_1-1}\ne 0.
\]
By Proposition~\ref{prop:2.1} (4) $-1/f\in \mathbf{N}_\kappa^{-k}$ and
\begin{equation}\label{eq:1/f}
  -\frac{1}{f(z)}=P_1(z)+o(1),\quad{as}\quad z\widehat{\rightarrow }\infty,
\end{equation}
where $P_1(z)$ is the Lanzcos polynomial of degree: $\deg P_1=n_1$. Since  ${\bf s}$ is regular, then $-l_1^{-1}:=P_1(0)\ne 0$ and hence~\eqref{eq:1/f} can be rewritten as
\begin{equation}\label{eq:1/f_2}
  -\frac{1}{f(z)}=zm_1(z)-\frac{1}{g_1(z)}=zm_1(z)-\frac{1}{l_1+f_1(z)}.
\end{equation}
Let $\kappa_1\!=\!\kappa_-(zm_1)$, $k_1\!=\!\kappa_-(m_1)$,  $k_1^+=k_1+\kappa_-(zl_1)$. Then
by Proposition~\ref{prop:2.1} (2), (4)
\[
zm_1\in\mathbf{N}_{\kappa_1}^{-k_1}\Rightarrow
-\frac{1}{g_1}=
-\frac{1}{f}- zm_1\in\mathbf{N}_{\kappa-\kappa_1}^{-(k-k_1)}\Rightarrow
g_1\in\mathbf{N}_{\kappa-\kappa_1}^{k-k_1}\Rightarrow
f_1\in\mathbf{N}_{\kappa-\kappa_1}^{k-k_1^+}.
\]
The formula~\eqref{eq:1/f_2} yields the following representation of $f\in\cM_\kappa^k({\bf s})$:
\begin{equation}\label{eq:ConFr1}
    f(z)=\displaystyle\frac{\displaystyle1}{\displaystyle -z m_{1}(z)+\displaystyle\frac{1}{\displaystyle
    l_{1}+f_1(z)}},
\end{equation}
where $f_1\in\mathbf{N}_{\kappa-\kappa_1}^{k-k_1^+}$ satisfies the asymptotic expansion
\begin{equation}\label{int4A}
    f_1(z)=-\frac{s^{(1)}_{0}}{z}-\frac{s^{(1)}_{1}}{z^2}-\cdots-\frac{s^{(1)}_{2n}}{z^{2n+1}}
    +o\left(\frac{1}{z^{2n+1}}\right),\quad\quad z\widehat{\rightarrow}\infty
\end{equation}
with some real  numbers $s^{(1)}_{i}$, $i\in\dZ_+$.
 Explicit formulae for calculation of the sequence $\mathbf{s}^{(1)}=\{s^{(1)}_{i}\}_{i=0}^\infty$ are presented in~\cite[Remarks 3.4, 3.6]{DK17}, the polynomial $m_1(z)$ and the number $l_1$ can be found by (see~\cite[(2.16), (3.27)]{DK17}):
\begin{equation}\label{eq:m_1}
        m_1(z)=\frac{(-1)^{n_1+1}}{D_{n_1}}
        \begin{vmatrix}
            0 & \ldots & 0 & s_{{n_1}-1} & s_{{n_1}} \\
            \vdots &  & \ldots & \ldots & \vdots \\
            s_{{n_1}-1} & \ldots & \ldots & \ldots & s_{2{n_1}-2} \\
            1 & z & \ldots& z^{{n_1}-2}& z^{{n_1}-1} \\
        \end{vmatrix}{,}
\end{equation}
\begin{equation}\label{eq:l_1}
l_1=(-1)^{{n_1}
    +1}s_{{n_1}-1}\frac{D_{n_{1}}}{D_{n_{1}}^+},\quad
    (D_\nu:=\det S_\nu,\,D_\nu^+:=\det S_\nu^+).
\end{equation}
 This completes the first step of the Schur algorithm.

 Applying this algorithm repeatidly $N$ times one obtains a function $f_N\in \mathbf{N}_{\kappa-\kappa_N}^{k-k_N^+}$, with $\kappa_N$ and $k_N$ given by~\eqref{eq:Gen_Solv+x}, connected with $f$ by the formula~\eqref{eq:con_fracx}.
 Moreover, the function $f_N$  satisfies the asymptotic expansion
\begin{equation}\label{eq:f_Nas}
    f_N(z)=-\frac{s^{(N)}_{0}}{z}-\frac{s^{(N)}_{1}}{z^2}-\cdots-\frac{s^{(N)}_{2n}}{z^{2n+1}}
    +o\left(\frac{1}{z^{2n+1}}\right),\quad\quad z\widehat{\rightarrow}\infty
\end{equation}
with some real $s^{(N)}_{i}$, $i\in\dZ_+$.

\begin{lemma} \label{lem:2.4}{\rm (\cite{DK15})}
 Let ${\mathbf s}\in\mathcal{H}_{\kappa}^{k,
reg}$. Then there exist  sequences of polynomials $m_j(z)$ and numbers  ${l}_j$ such that
the $2j-$th convergent $\frac{u_{2j}}{v_{2j}}$ of the continued fraction
\begin{equation}\label{s4.2}
    \displaystyle\frac{\displaystyle1}{\displaystyle -z m_{1}(z)+\displaystyle\frac{1}{\displaystyle
    l_{1}+\dots\frac{1}{\displaystyle -z m_{j}(z) +\frac{1}{\displaystyle
    l_{j}+\dots}}}}
\end{equation}
coincides with the $j-$th convergent of the $P-$fraction \eqref{eq:Pfrac} corresponding to the
sequence ${\mathbf s}$.  Let  functions $f$ and $f_N$ ($N\in\dN$) be connected by~\eqref{eq:con_fracx} and let $\mathbf{s}^{(N)}$ is the $N-$th induced sequence. Then:
\begin{equation}\label{eq:InducedMP_Tr}
  f\in\cM_\kappa^k(\mathbf{s},\ell) \Longleftrightarrow  f_N\in\cM_{\kappa-\kappa_N}^{k-k_N^+}(\mathbf{s}^{(N)},\ell-2n_N),
\end{equation}
\begin{equation}\label{eq:InducedMP}
  f\in\cM_\kappa^k(\mathbf{s}) \Longleftrightarrow  f_N\in\cM_{\kappa-\kappa_N}^{k-k_N^+}(\mathbf{s}^{(N)}),
\end{equation}
where $\kappa_N$ and $k_N^+$ are  given by~\eqref{eq:Gen_Solv+x}.
\end{lemma}

In the  case of a regular sequence $\mathbf{s}\in\cH_\kappa^k$ the parameters ${l}_j$ and $m_j(z)$ in~\eqref{s4.2} can be calculated
recursively by the above  Schur algorithm in terms of the sequence  ${\mathbf s}$:
\begin{equation}\label{eq:ml_j}
        m_j(z)=\frac{(-1)^{\nu+1}}{D^{(j-1)}_{\nu}}
        \begin{vmatrix}
            0 & \ldots & 0 & s^{(j-1)}_{\nu-1} & s^{(j-1)}_{\nu} \\
            \vdots &  & \ldots & \ldots & \vdots \\
            s^{(j-1)}_{\nu-1} & \ldots & \ldots & \ldots & s^{(j-1)}_{2\nu_{}-2} \\
            1 & z & \ldots& z^{\nu-2}& z^{\nu-1} \\
        \end{vmatrix},
\end{equation}
where $D^{(j)}_\nu:=\det S^{(j)}_\nu$, $\nu=n_{j}-n_{j-1}$ and
\begin{equation}\label{eq:l_j0}
l_j
   =(-1)^{\nu+1}\frac{D^{(j-1)}_{\nu}}{\left(D^{(j-1)}_{\nu}\right)^+}\quad (j=1,\dots,N-1).
\end{equation}

The $N-$th induced sequence $\mathbf{s}^{(N)}$ can be found as the sequence of  coefficients of the series expansion $-\sum_{i=0}^\infty{s_i^{(N)}}{z^{-(i+1)}}$ corresponding to the continued fraction
 \begin{equation}\label{s8.0}
    \frac{1}{\displaystyle -zm_{N+1}(z)+\frac{1}{\displaystyle
    l_{N+1}+\frac{1}{\displaystyle -zm_{N+2}(z)+\frac{1}{\ddots} }}}.
\end{equation}

\subsection{Generalized Stieltjes  continued fractions}
 In~\cite{DK15} the  expansion~\eqref{s4.2} of $f\in\cM_\kappa^k(\mathbf{s})$ into a generalized Stieltjes fraction for $\textbf{s}\in\cH_{\kappa}^{k,reg}$ was derived from the expansion of its unwrapping transform $zf(z^2)$ into the $P-$fraction.

\begin{theorem}  \label{thm:5.1}{\rm (\cite{DK15})}
 Let ${\mathbf s}\in\mathcal{H}_{\kappa}^{k,
reg}$ and let the $P-$fraction \eqref{eq:Pfrac} and the generalized $S-$fraction (\ref{s4.2})
correspond to the sequence ${\mathbf s}$. Then the parameters ${l}_j$ and $m_j(z)$
($j\in\mathbb{Z}_+$) of the generalized $\mathbf{S}-$fraction (\ref{s4.2})
are connected with the parameters $b_j$ and $a_j(z)$
($j\in\mathbb{N}$) of the $P-$fraction \eqref{eq:Pfrac} by the
equalities
\begin{equation}\label{s4.9}
    b_{0}=\frac{1}{d_{1}},\quad a_{0}(z)=\frac{1}{d_{1}}\left(zm_{1}(z)-\frac{1}{l_{1}}\right),
\end{equation}
\begin{equation}\label{s4.10}
     b_{j}=\frac{1}{l_{j}^{2}d_{j}d_{j+1}},\quad
     a_{j}(z)=\frac{1}{d_{j+1}}\left(zm_{j+1}(z)-\left(\frac{1}{l_{j}}
    +\frac{1}{l_{j+1}}\right)\right),
\end{equation}
where $d_{j}$ is the leading coefficient of $m_{j}(z)$ $(j=1,\dots,N-1)$.
\end{theorem}
In the case when  $m_j(z)\equiv m_j$ are constants the generalized $S-$fraction reduces to the classical $S$-fraction~\eqref{s4.2Cx}
and the formulae \eqref{s4.9} and \eqref{s4.10} coincide with the well known classical formulae from \cite{St94}, see also~\cite[Appendix, (3), (4)]{Akh}.

Conversely, $m_j$ and
$l_j$ can be represented in terms of the $P-$fractions, see~\cite[Corollary 4.1]{DK15}.
\begin{corollary}
\label{l.2.6}
  Let ${\mathbf s}\in\mathcal{H}_{\kappa}^{k,reg}$, let ${\mathbf s}$ be associated with the $\mathbf{S}-$fraction \eqref{s4.2}, let $d_i$ be the leading coefficient of the polynomial $m_i(z)$  and let
\begin{equation}\label{eq:wtbj}
  \widetilde{b}_{i}=b_0b_1\ldots b_{i},\quad i\in\dZ_+.
\end{equation}
Then:
    \begin{equation}\label{l.2.6.we1}
        d_i=\frac{P_{i-1}^2(0)}{\widetilde{b}_{i-1}},\quad
         l_i=-\frac{\widetilde{b}_{i-1}}{P_{i-1}(0)P_i(0)},\quad
         m_i(z)=d_i\frac{a_{i-1}(z)-a_{i-1}(0)}{z},\quad i\in\dN.
     \end{equation}
\end{corollary}
\begin{remark} If $\cN(\mathbf{s})=\dN$, then $\textup{deg}(m_i)=0$ for all $i\in\dN$, and $d_i=m_i$, $l_i$ can be found by~\eqref{l.2.6.we1}. 
Alternatively, $m_i$, $l_i$ can be found by~\eqref{int5x}.

If $D_i^+> 0$ for all $i\in\dN$, then $\textup{deg}(m_i)\le 1$ for all $i\in\dN$, and $m_i(z)=d_iz+m_i$, $l_i$ can be found by~\eqref{l.2.6.we1}. The parameters $d_i$, $m_i$ and $l_i$ can be also expressed in terms of $D_i$ and $D_i^+$, see~\cite[Section 5.3]{KL81}, \cite[Remark 4.1]{DK15}.
\end{remark}

\section{Truncated indefinite moment problems}\label{sec:3}


\subsection{A system of difference equations and generalized Stieltjes  polynomials}\label{sec:3.1}

Let us consider a system of difference equations associated with the continued fraction (\ref{s4.2})
\begin{equation}\label{s4.3}
    \left\{
    \begin{array}{l}
        y_{2j-1}-y_{2j-3}=-zm_{j}(z)y_{2j-2}{,}\\
        y_{2j}-y_{2j-2}=l_{j}y_{2j-1},\\
    \end{array}\right.\quad j\in\dN.
\end{equation}
If the $j$--th convergent of this continued fraction is denoted by
$\frac{u_{j}}{v_{j}}$, then $u_{j}$, $v_{j}$ can be found as solutions of the
system~\eqref{s4.3}  (see~\cite[Section~1]{Wall})
subject to the following initial conditions
\begin{equation}\label{s4.4}
    u_{-1}\equiv1, \quad u_{0}\equiv0; \qquad v_{-1}\equiv0, \quad v_{0}\equiv1.
\end{equation}
The first two convergents  of the continued fraction (\ref{s4.2}) take the form
\[
   \label{s4.5}
    \frac{u_{1}}{v_{1}}=\frac{1}{-zm_{1}(z)}=T_{M_1}[\infty], \quad
    \frac{u_{2}}{v_{2}}=\frac{l_{1}}{-zl_{1}m_{1}(z)+1}
    =T_{M_1L_1}[0].
\]
\begin{definition}(\cite{DK16})\label{def:St_Pol}
Let $\textbf{s}\in \mathcal{H}_{\kappa}^{k,reg}$. Define
polynomials $P^{+}_{i}(z)$, $Q_{i}^{+}(z)$ by
\begin{equation}\label{2p.new8.r7}
    \begin{split}
        &P^{+}_{-1}(z)\equiv0,\quad P^{+}_{0}(z)\equiv1,\qquad
        Q^{+}_{-1}(z)\equiv1,\quad Q^{+}_{0}(z)\equiv0,
        \\
        &P^{+}_{2i-1}(z)=-\frac{1}{\widetilde{b}_{i-1}}
        \begin{vmatrix}
            P_{{i}}(z) & P_{{i-1}}(z) \\
            P_{{i}}(0) & P_{{i-1}}(0)\\
        \end{vmatrix}\quad\mbox{and}
         \quad P^{+}_{2i}(z)=\frac{P_{{i}}(z)}{P_{{i}}(0)},\\&
        Q_{2i-1}^{+}(z)=\frac{1}{ \widetilde{b}_{i-1}}
        \begin{vmatrix}
            Q_{{i}}(z) & Q_{{i-1}}(z) \\
            P_{{i}}(0) & P_{{i-1}}(0)\\
        \end{vmatrix}\quad\mbox{and}
         \quad
        Q^{+}_{2i}(z)=-\frac{Q_{{i}}(z)}{P_{{i}}(0)}.
    \end{split}
\end{equation}
The polynomials $P^{+}_{i}(z)$, $Q_{i}^{+}(z)$ are called \emph{ the  generalized
Stieltjes  polynomials} corresponding to the sequence ${\mathbf s}$.
\end{definition}
As was noticed in~\cite{DK16}  the  generalized Stieltjes  polynomials coincide with the solutions $u_i$ and $v_i$ of the system~\eqref{s4.3}.
\begin{proposition}\label{prop:5.4}
Let $\textbf{s}\in \mathcal{H}_{\kappa}^{k,reg}$ and let $P^{+}_{j}(z)$ and $Q_{j}^{+}(z)$
be the  generalized  Stieltjes  polynomials defined by~\eqref{2p.new8.r7}.  Then the solutions $\{u_j\}_{j=0}^N$ and $\{v_j\}_{j=0}^N$ of the system~\eqref{s4.3}, \eqref{s4.4} take the form
\[
        u_{j}= Q_{j}^+(z),\quad
        v_{j}= P_{j}^+(z) 
        \quad(j=-1,0,\dots,N).
\]
\end{proposition}

\begin{remark} The Stieltjes polynomials satisfy the following
properties
\begin{equation}\label{6p.eq:propStpol1}
        P^{+}_{2i-1}(0)=0,\quad P^{+}_{2i-2}(0)=1\quad \mbox{and}\quad
        Q^{+}_{2i-1}(0)=1.
\end{equation}

    Obviously, by Definition \ref{def:St_Pol}
\begin{equation}\label{6p.eq:propStpol2}
    P^{+}_{2i-1}(0)=-\frac{1}{\widetilde{b}_{i-1}}
        \begin{vmatrix}
            P_{{i}}(0) & P_{{i-1}}(0) \\
            P_{{i}}(0) & P_{{i-1}}(0)\\
        \end{vmatrix}=0
    \quad\mbox{and}\quad
    P^{+}_{2i-2}(0)=\frac{P_{{i}}(0)}{P_{{i}}(0)}=1.
\end{equation}
        By \eqref{6p.eq:gLN1}
\begin{equation}\label{6p.eq:propStpol2x}
    Q^{+}_{2i-1}(0)=\frac{1}{\widetilde{b}_{i-1}}
        \begin{vmatrix}
            Q_{{i}}(0) & Q_{{i-1}}(0) \\
            P_{{i}}(0) & P_{{i-1}}(0)\\
        \end{vmatrix}=\frac{Q_{{i}}(0)P_{{i-1}}(0)-Q_{{i-1}}(0)P_{{i}}(0)}{\widetilde{b}_{i-1}}=1.
\end{equation}
{
Here we used the  generalized Liouville-Ostrogradsky formula (see~\cite[(2.9)]{DK17})
\begin{equation}\label{6p.eq:gLN1}
    Q_{{i}}(z)P_{{i-1}}(z)-Q_{{i-1}}(z)P_{{i}}(z)=\widetilde{b}_{i-1},\quad i\in\dN.
\end{equation}
}
\end{remark}

\begin{lemma}\label{6p.lem:stlpol1}
    Let $P^{+}_{i}$ and $Q^{+}_{i}$ be the Stieltjes polynomials
defined by \eqref{2p.new8.r7}. Then
\begin{equation}\label{6p.relfor STpol1}
   P^{+}_{2i}(z)Q^{+}_{2i-1}(z)- Q^{+}_{2i}(z)P^{+}_{2i-1}(z)=1.
\end{equation}
\end{lemma}
\begin{proof}
    By Definition \ref{def:St_Pol} and~\eqref{6p.eq:gLN1} we obtain
\[
    \begin{split}
                P^{+}_{2i}(z)Q^{+}_{2i-1}(z)-Q^{+}_{2i}(z)P^{+}_{2i-1}(z)=&
                \frac{P_{{i}}(z)}{\widetilde{b}_{i-1}P_{{i}}(0)}(Q_{{i}}(z)P_{{i-1}}(0)-P_{{i}}(0)Q_{{i-1}}(z))-\\&-
                \frac{Q_{{i}}(z)}{
                \widetilde{ b}_{i-1}P_{{i}}(0)}(P_{{i}}(z)P_{{i-1}}(0)-P_{{i}}(0)P_{{i-1}}(z))\\
                =&\frac{Q_{{i}}(z)P_{{i-1}}(z)-P_{{i}}(z)
                Q_{{i-1}}(z)}{\widetilde{b}_{i-1}}=1
    \end{split}
\]
 This completes the proof. \end{proof}

\begin{lemma}\label{6p.lem_l_sat}
    Let  $\textbf{s}\in \mathcal{H}_{\kappa}^{k,reg}$ and let $P_i(z)$  and $Q_i(z)$ ($i\in\dZ_+$)  be Lanczos polynomials of the first and second kind and let ${l}_j$ and $m_j(z)$ ($j\in\mathbb{N}$) be parameters  of the generalized $\mathbf{S}-$fraction (\ref{s4.2}).
    Then:
\begin{enumerate}
  \item [(i)]  The constants $l_i$ can be calculated by
\begin{equation}\label{6p.eq:lem_l_1}
    l_i=-\frac{Q_i(0)}{P_i(0)}+\frac{Q_{i-1}(0)}{P_{i-1}(0)},\quad i\in\dN;
\end{equation}
  \item [(ii)] For every $N\in\dN$ the following formulas hold
 {
  \begin{equation}\label{6p.eq:cor.l1x}
    \sum_{i=1}^{N}l_i=-\frac{Q_{N}(0)}{P_{N}(0)},\qquad
            \sum_{i=1}^{N}
            d_{i}=\sum_{i=0}^{N-1}|P^2_{i}(0)|\widetilde{b}_i^{-1};
        \end{equation}
}
      \begin{equation}\label{6p.eq:cor.l1}
    \sum_{i=1}^{N}m_{i}(0)=-P^{+^{'}}_{2N-1}(0).
    \end{equation}
\end{enumerate}
\end{lemma}
\begin{proof}
1)  Let $Q^{+}_{i}(z)$   be Stieltjes
    polynomials defined by~\eqref{2p.new8.r7}.
    Substituting in \eqref{s4.3} $y_j=Q_j^+$ and $z=0$, we obtain
\[
    Q^{+}_{2i}(0)-Q^{+}_{2i-2}(0)=l_{i}Q^{+}_{2i-1}(0).
\]
By Definition \ref{def:St_Pol} and by the generalized Liouville-Ostrogradsky
formula~\eqref{6p.eq:gLN1}
\[
    Q^{+}_{2i-1}(0)=\frac{1}{\wt b_{i-1}}(Q_{{i}}(0)P_{{i-1}}(0)-Q_{{i-1}}(0)P_{{i}}(0)=1.
\]
This implies~\eqref{6p.eq:lem_l_1}.

{
2) Summing the equalities~\eqref{6p.eq:lem_l_1} for $i=1,\dots, N$ one obtains the first equality in~\eqref{6p.eq:cor.l1x}.
The second equality in~\eqref{6p.eq:cor.l1x} is implied by the relation (see     \cite[Corollary 4.1]{DK15})
        \begin{equation}\label{6p.P(0)_2}
            d_{i+1}=|P^2_{i}(0)|\widetilde{b}_i^{-1}.
        \end{equation}

3)    Differentiating the first equality in~\eqref{s4.3} one  obtains
    \[
        P^{+^{'}}_{2i-1}(z)=-m_{i}(z)P^{+}_{2i-2}(z)-z(m_{i}(z)P^{+}_{2i-2}(z))'+P^{+^{'}}_{2i-3}(z).
    \]
    Substituting $z=0$ and using the equality $P_{2i-2}^+(0)=1$, one obtains
    \[
        m_{i}(0)=
        \,P^{+^{'}}_{2i-3}(0)-P^{+^{'}}_{2i-1}(0).
    \]
}
   Hence
    \[
       \sum_{i=1}^{N}m_{i}(0)=P^{+^{'}}_{-1}(0)-P^{+^{'}}_{1}(0)
    +\ldots+P^{+^{'}}_{2N-3}(0)-P^{+^{'}}_{2N-1}(0)
    =-P^{+^{'}}_{2N-1}(0).
    \]
  This  proves \eqref{6p.eq:cor.l1}.
\end{proof}

\subsection{The class $\mathcal{U}_\kappa^k(J)$ and linear fractional transformations}
Let  $J$ and $\cZ$ be the  $2\times 2$
 matrices
\[
J=\begin{pmatrix}
   0  &  -i \\
   i  &  0\\
\end{pmatrix}\quad\mbox{and}\quad
\cZ=\begin{pmatrix}
   z &  0 \\
   0  &  1 \\
\end{pmatrix}.
\]
{
\begin{definition}\label{def:Jkk}
Let  $W(z)$ be a $2\times 2$ matrix valued function meromorphic in $\mathbb{C}_+$ and let ${\gh}_W^+$ be the domain of holomorphy of $W$ in $\mathbb{C}_+$, $\kappa\in\dZ_+$. Then $W(z)$ is called a {\it generalized $J$-inner} matrix valued function from the class ${\mathcal{U}}_\kappa(J)$,
if:

\begin{enumerate}
\item[(i)]
the kernel
\begin{equation}\label{kerK}
{\mathsf K}_\omega^W(z)=
\frac{J-W(z)JW(\omega)^*}{-i(z-\bar \omega)}
\end{equation}
has $\kappa$ negative squares in ${\gh}_W^+\times{\gh}_W^+$;
\item[(ii)]
$J-W(\mu)JW(\mu)^*=0$ for a.e.  $\mu\in\mathbb{R}$.
\end{enumerate}
A matrix valued function $W\in \cU_\kappa(J)$ is said to belong to the class $\cU_\kappa^k(J)$, $\kappa,k\in\dZ_+$, if
\begin{equation}\label{eq:Ukk}
  \cZ W\cZ^{-1}\in \cU_k(J).
\end{equation}
\end{definition}

Consider the linear fractional transformation
\begin{equation}\label{eq:0.9}
    T_W[\tau]=(w_{11}\tau(z)+w_{12})(w_{21}\tau(z)+w_{22})^{-1}
\end{equation}
associated with the matrix valued function $W(z)=(w_{i,j}(z))_{i,j=1}^2$. The linear fractional transformation associated with the product $W_1W_2$ of two matrix valued functions $W_1(z)$ and $W_2(z)$, coincides with the composition $T_{W_1}\circ T_{W_2}$.
}

The following statement is an easy corollary of Definition~\ref{def:Jkk}.
\begin{lemma}\label{lem:W}
Let $W\in{\mathcal{U}}_{\kappa_1}^{k_1}(J)$ and $\tau\in{\mathbf N}_{\kappa_2}^{k_2}$, $\kappa_1,\kappa_2,k_1,k_2\in\dZ_+$. Then $T_W[\tau]\in{\mathbf N}_{\kappa'}^{k'}$, where $\kappa'\le{\kappa_1+\kappa_2}$, $k'\le k_1+k_2$.
\end{lemma}
\begin{proof}
  The proof of the inclusion $T_W[\tau]\in{\mathbf N}_{\kappa'}$ for some $\kappa'\le{\kappa_1+\kappa_2}$ is similar to that in~\cite[Lemma 3.4]{DD09}.

  Let us denote $f(z)=T_W(z)[\tau(z)]$. Then
  \[
  zf(z)=T_{\cZ W(z)\cZ^{-1}}[z\tau(z)].
  \]
  Since $\cZ W\cZ^{-1}\in \cU_{k_1}$ and $z\tau\in {\mathbf N}_{k_2}$ then $f\in {\mathbf N}_{k'}$ for some
  $k'\le k_1+k_2$.
  This proves that $T_W[\tau]\in{\mathbf N}_{\kappa'}^{k'}$ for some $\kappa'\le{\kappa_1+\kappa_2}$, $k'\le k_1+k_2$.
\end{proof}

In the present paper two special types of matrix valued functions from $\cU_\kappa^k(J)$ play an important role (see~\cite[Lemma~2.11, Lemma~2.12]{DK17}).
\begin{lemma}\label{lem:W_my}
Let $m(z)$ be a real polynomial such that $\kappa_-(zm(z))=\kappa_1$, $\kappa_-(m(z))=k_1$, let $M(z)$ be a $2\times 2$ matrix valued function
\begin{equation}\label{eq:W_m}
    M(z)=\begin{pmatrix}
   1  &  0 \\
   -z m(z)  &  1\\
\end{pmatrix}
\end{equation}
and let $\tau$ be a meromorphic function, such that
\begin{equation}\label{eq:Nev_tau}
  \tau(z)^{-1}=o(z)\mbox{ as }z\widehat{\rightarrow}\infty.
\end{equation}
Then $M\in\cU_{\kappa_1}^{k_1}(J)$ and the following equivalences hold:
\begin{equation}\label{eq:w_m_k}
    \tau\in{\mathbf N}_{\kappa_2}\Longleftrightarrow T_M[\tau]\in {\mathbf N}_{\kappa_1+\kappa_2},
\end{equation}
\begin{equation}\label{eq:w_m_kk}
    \tau\in{\mathbf N}_{\kappa_2}^{k_2}\Longleftrightarrow  T_M[\tau]\in {\mathbf N}_{\kappa_1+\kappa_2}^{k_1+k_2}.
\end{equation}
\end{lemma}

\begin{lemma}\label{lem:W_l}
Let $l(z)$ be a real polynomial such that $\kappa_-(l(z))=\kappa_1$,
$\kappa_-(zl(z))=k_1$, let $L(z)$ be the $2\times 2$ matrix valued
function
\begin{equation}\label{eq:W_lx}
    L(z)=\begin{pmatrix}
   1  &  l(z) \\
   0  &  1\\
\end{pmatrix}
\end{equation}
and let $\phi$ be a meromorphic function, such that
\begin{equation}\label{eq:Nev_phi}
  \phi(z)=o(1)\mbox{ as }z\widehat{\rightarrow}\infty.
\end{equation}
Then $L\in\cU_{\kappa_1}^{k_1}(J)$ and the following equivalences hold:
\begin{equation}\label{eq:w_l_k}
    \phi\in{\mathbf N}_{\kappa_2}\Longleftrightarrow T_L[\phi]\in {\mathbf N}_{\kappa_1+\kappa_2},
\end{equation}
\begin{equation}\label{eq:w_l_kk}
    \phi\in{\mathbf N}_{\kappa_2}^{k_2}\Longleftrightarrow T_L[\phi]\in {\mathbf N}_{\kappa_1+\kappa_2}^{k_1+k_2}.
\end{equation}
\end{lemma}


\subsection{Truncated  indefinite moment problem and resolvent matrices} \label{sec:3.2}
As was mentioned in Theorem~\ref{thm:4.2} the matrix valued function $W_{2N}(z)$ defined by~\eqref{eq:ResM_2N} is the resolvent matrix  of a truncated even indefinite moment problem $MP_{\kappa}^k(\textbf{s},  2N-1)$ with $\kappa\ge \kappa_N$ and $k\ge k_N^+$, where $\kappa_N$ and $k_N^+$ are defined by~\eqref{eq:Gen_Solv+x}.
The resolvent matrix $W_{2N}(z)$  provides an example of generalized $J-$inner matrix valued function from the class $\cU_{\kappa_N}^{k_N^+}(J)$.

\begin{theorem}\label{thm:4.2A}{\rm (\cite{DK17})}
Let ${\mathbf s}=\{s_i\}_{i=0}^{\infty}\in\mathcal{H}_{\kappa}^{k,reg}$, let   $P_{{j}}^+(z)$ and
$Q_{{j}}^+(z)$  be generalized Stieltjes polynomials  of the first and second kind, let $m_i(z)$, $l_i$ be defined by~\eqref{eq:ml_j} and~\eqref{eq:l_j0} and let $W_{2N}(z)$, $M_i(z)$ and $L_i$  be given by~\eqref{eq:ResM_2N} and~\eqref{eq:LMjIx}.
 Then:
\begin{enumerate}
  \item [(i)] The matrix valued function  $W_{2N}(z)$ admits the factorization
  \begin{equation}\label{eq:W1j}
    W_{2N}(z)=M_1(z)L_1\dots M_N(z)L_N;
\end{equation}
\item [(ii)] $W_{2N}\in\cU_{\kappa_N}^{k_N^+}(J)$.
\end{enumerate}
\end{theorem}
\begin{proof}
  It follows from \eqref{s4.3} that
\begin{equation}\label{eq:W2N-2}
   W_{2N}(z)=W_{2N-2}(z)M_N(z)L_N.
\end{equation}
Applying \eqref{eq:W2N-2} $N$ times one obtains~\eqref{eq:W1j}.

By \cite[Theorem 4.3 (4)]{DK17}
\[
    \begin{split}
    &\kappa_{N}=\sum\limits_{j=1}^{N}\kappa_{-}(zm_{j})
    ,\quad
    k_{N}^+=\sum\limits_{j=1}^{N}k_{-}(m_{j})+
    \sum\limits_{j=1}^{N}\kappa_{-}(zl_{j}).
    \end{split}
\]
Since by Lemmata~\ref{lem:W_my}, \ref{lem:W_l} $M_i\in\cU_{\kappa_-(zm_i)}^{\kappa_-(m_i)}$, $L_i\in\cU_{0}^{\kappa_-(zl_i)}$, $i\in\dN$,
one obtains from the  factorization formula~\eqref{eq:W1j} $W_{2N}\in\cU_{\kappa'}^{k'}(J)$, where
\begin{equation}\label{eq:kappaN>}
  \kappa'\le \kappa_{N}, \quad k'\le k_{N}^+.
\end{equation}
On the other hand it follows from Lemma~\ref{lem:W} that $f:=T_{W_{2N}[0]}\in\mathbf{N}_{\kappa_N}^{k_{N}^+}$.
Therefore, by Lemma~\ref{lem:W}
\begin{equation}\label{eq:kappaN<}
  \kappa_{N}\le\kappa', \quad k_{N}^+\le k'.
\end{equation}
The statement (ii) follows from~\eqref{eq:kappaN>} and \eqref{eq:kappaN<}.
\end{proof}

Let us formulate an analog  of Theorems~\ref{thm:4.2} and~\ref{thm:4.2A} for the odd truncated  { in\-de\-fi\-nite} moment problem (see~\cite{DK17}).

\begin{theorem}\label{thm:DescrOdd}
Let ${\mathbf
s}=\{s_i\}_{i=0}^{2n_N-2}\in\cH_{\kappa}^{k,reg}$,
let $P_j^+(z)$ and $Q_j^+(z)$ $(0\le j\le 2N-1)$ be
generalized Stieltjes polynomials, let $M_i(z)$ and $L_i$  be given by~\eqref{eq:LMjIx} and let the matrix valued function $W_{2N-1}$ be defined by
\begin{equation}\label{eq:W_2N-1}
W_{2N-1}(z)=
    \begin{pmatrix}
        Q^{+}_{2N-1}(z) & Q^{+}_{2N-2}(z) \\
        P^{+}_{2N-1}(z) & P^{+}_{2N-2}(z) \\
    \end{pmatrix}.
\end{equation}
 Then:
\begin{enumerate}
  \item [(i)] The odd moment problem
$MP_{\kappa}^{k}(\textbf{s},2n_N-2)$ is solvable, if and only if
\begin{equation}\label{eq:Gen_Solv'''}
    \kappa_N:=\nu_-(S_{n_N})\le\kappa\quad\mbox{and}\quad
    k_N:=\nu_-(S_{n_N-1}^+)\le k.
\end{equation}
\item [(ii)] $f\in \mathcal{M}_{\kappa}^{k}(\textbf{s},
  2n_{N}-2)$ if and only if $f$ admits the
  representation
\begin{equation}\label{2p.new6}
    f(z)=T_{W_{2N-1}(z)}[\tau(z)]
    =\frac{Q^{+}_{2N-1}(z)\tau(z)+Q^{+}_{2N-2}(z)}{P^{+}_{2N-1}(z)\tau(z)+P^{+}_{2N-2}(z)},
\end{equation}
where
$    \tau\in N^{k-k_{N}}_{\kappa-\kappa_{N}}$ and ${\tau(z)}^{-1}=o(z)$ as
    $z\widehat{\rightarrow}\infty.$
      \item [(iii)] The matrix valued function  $W_{2N-1}(z)$ belongs to the class $\cU_{\kappa_N}^{k_N}(J)$ and admits the factorization
  \begin{equation}\label{eq:W1j_B}
    W_{2N-1}(z)=M_1(z)L_1\dots L_{N-1} M_N(z).
\end{equation}
\end{enumerate}
\end{theorem}

{ Substituting in \eqref{2p.new6} $\tau(z)=\infty$ one obtains from Theorem~\ref{thm:DescrOdd} the following:
\begin{corollary}\label{cor:AC}
  The function $\frac{ Q^{+}_{2N-1}(z) }{ P^{+}_{2N-1}(z) }$ belongs to the class ${\mathbf N}_{\kappa_{N}}^{k_{N}}$.
\end{corollary}
Similarly, substituting in \eqref{eq:LFT_W} $\tau(z)\equiv 0$ one obtains from Theorem~\ref{thm:4.2}.}
\begin{corollary}\label{cor:BD}
  The function $\frac{ Q^{+}_{2N}(z) }{ P^{+}_{2N}(z) }$ belongs to the class ${\mathbf N}_{\kappa_{N}}^{k_{N}^+}$.
\end{corollary}

\subsection{Two lemmata about linear fractional transformations $T_M$ and $T_L$}\label{sec:3.5}
The statements of Lemmata~\ref{lem:W_my} and~\ref{lem:W_l} fail to hold if the conditions
{\begin{equation}\label{eq:Nev_tau1}
\tau(z)^{-1}=o(z)\mbox{ as }z\widehat{\rightarrow}\infty,
\end{equation}
\begin{equation}\label{eq:Nev_phi1}
  \phi(z)= o(1)\quad\mbox{ as }z\widehat{\rightarrow}\infty
\end{equation}}
 are not satisfied. In these cases  the number of moments interpolating by the linear fractional transformations $T_W[\tau]$   can be reduced and also their indices  can decrease.
We will start with the linear fractional transformations $T_M$ in the simplest cases when $m$ is a positive constant and hence $M\in\cU_0^0(J)$.
\begin{lemma}\label{lem:W_m2}
Let $M(z)$ be a $2\times 2$ matrix valued function
\begin{equation}\label{eq:W_m2}
    M(z)=\begin{pmatrix}
   1  &  0 \\
   -z m  &  1\\
\end{pmatrix},{\quad m>0,}
\end{equation}
let $\tau\in  {\mathbf N}_{\kappa}^{k}$ be a  function from $ {\mathbf N}_{\kappa}^{k}$, such that
$\tau(z)^{-1}\ne o(z)\mbox{ as }z\widehat{\rightarrow}\infty$ and let $\phi=T_{M}[\tau]$.
Then:
\begin{enumerate}
  \item [({ i})] either  $\phi\in  {\mathbf N}_{\kappa}^{k}$ and $\phi(z)=o(1)\mbox{ as }z\widehat{\rightarrow}\infty$,
  \item [({ ii})] or  $\phi\in  {\mathbf N}_{\kappa-1}^{k}$.
\end{enumerate}
\end{lemma}
\begin{proof}
{\bf { 1}.} {\it Verification of ${ (i)}$ for $\tau\in  {\mathbf N}_{\kappa}^{k}$, such that$:$}
\begin{equation}\label{eq:3.11}
  \lim_{z\widehat{\rightarrow}\infty}\frac{-1}{z\tau(z)}=-\infty.
\end{equation}
By Lemma~\ref{prop:2.1} $-\tau^{-1}\in  {\mathbf N}_{\kappa}^{-k}$.
If~\eqref{eq:3.11} holds then $mz-\tau(z)^{-1}$ has GPNT (generalized pole of negative type) at $\infty$ of the same multiplicity as $-\tau(z)^{-1}$, i.e.
\begin{equation}\label{eq:Pmult1}
  {\kappa_\infty\left(mz-\frac{1}{\tau(z)}\right)=\kappa_\infty\left(-\frac{1}{\tau(z)}\right).}
  \end{equation}
Since also
\begin{equation}\label{eq:Pmult2}
\kappa_\infty\left(m-\frac{1}{z\tau(z)}\right)=\kappa_\infty\left(-\frac{1}{z\tau(z)}\right).
\end{equation}
one obtains by Theorem~\ref{prop:GPNT}
$mz-\tau(z)^{-1}\in  {\mathbf N}_{\kappa}^{-k}$. Hence $\phi(z)=\frac{-1}{mz-\tau(z)^{-1}}\in  {\mathbf N}_{\kappa}^{k}$ by Lemma~\ref{prop:2.1} and, moreover, $\phi(z)=o(1)$  as $z\widehat{\rightarrow}\infty$, since
\[
  \lim_{z\widehat{\rightarrow}\infty}\phi(z)=
  \lim_{z\widehat{\rightarrow}\infty}\frac{1/z}{m-(z\tau(z))^{-1}}=0.
\]

{\bf 2.} {\it Verification of $({ i})$ for $\tau\in  {\mathbf N}_{\kappa}^{k}$, such that$:$}
\begin{equation}\label{eq:3.13}
  \lim_{z\widehat{\rightarrow}\infty}\frac{-1}{z\tau(z)}=a 
\end{equation}
and $a\ge 0$. In this case $\kappa_\infty(-\tau(z)^{-1})=0$ and  $-\tau(z)^{-1}$ admits the representation
\begin{equation}\label{eq:3.14}
 -\tau(z)^{-1}=az-\tau_1(z)^{-1},
 \end{equation}
 where  $\lim\limits_{z\widehat{\rightarrow}\infty}\frac{-1}{z\tau_1(z)}=0$.
 Then
the function
 $mz-\tau(z)^{-1}=(m+a)z-\tau_1(z)^{-1}$ has no  GPNT at $\infty$ since
\[
\lim_{z\widehat{\rightarrow}\infty}\left(m-\frac{1}{z\tau(z)}\right)=m+a>0.
\]
 Then~\eqref{eq:Pmult1} and \eqref{eq:Pmult2} hold and by Theorem~\ref{prop:GPNT}
 $mz-\tau(z)^{-1}\in  {\mathbf N}_{\kappa}^{-k}$. Moreover,
\begin{equation}\label{eq:3.16}
  \lim_{z\widehat{\rightarrow}\infty}\phi(z)=
  \lim_{z\widehat{\rightarrow}\infty}\frac{1/z}{m+a-(z\tau(z))^{-1}}=0.
\end{equation}
and hence $({ i})$ holds.

{\bf 3.} {\it Verification of $({ i})$ for $\tau\in  {\mathbf N}_{\kappa}^{k}$, such that
\eqref{eq:3.13} holds and $a<-m$ $:$}
In this case
\[
\lim_{z\widehat{\rightarrow}\infty}\left(m-\frac{1}{z\tau(z)}\right)=m+a<0
\]
and hence
\[
  \kappa_\infty\left(mz-\frac{1}{\tau(z)}\right)=\kappa_\infty\left(-\frac{1}{\tau(z)}\right)=1.
  \]
Since also~\eqref{eq:Pmult2}
holds then
by Theorem~\ref{prop:GPNT}
 $mz-\tau(z)^{-1}\in  {\mathbf N}_{\kappa}^{-k}$, and  by \eqref{eq:3.16} $\phi(z)=o(1)$   as $z\widehat{\rightarrow}\infty$. Hence (1) holds.

{\bf 4.} {\it Verification of $({ ii})$ for $\tau\in  {\mathbf N}_{\kappa}^{k}$, such that
\eqref{eq:3.13} holds and $a\in[-m,0)$ $:$}
In this case
\[
  \kappa_\infty\left(mz-\frac{1}{\tau(z)}\right)=0,\quad \kappa_\infty\left(-\frac{1}{\tau(z)}\right)=1
  \]
and since \eqref{eq:Pmult2} holds then  $mz-\tau(z)^{-1}\in  {\mathbf N}_{\kappa-1}^{-k}$ by Theorem~\ref{prop:GPNT}. Therefore, $\phi(z)=\frac{-1}{mz-\tau(z)^{-1}}\in  {\mathbf N}_{\kappa-1}^{k}$ by Lemma~\ref{prop:2.1}.
This proves $({ ii})$.
\end{proof}
\begin{lemma}\label{lem:W_l2}
Let  $L$ be the $2\times 2$ matrix
\begin{equation}\label{eq:W_l}
    L=\begin{pmatrix}
   1  &  l \\
   0  &  1\\
\end{pmatrix}, \quad\mbox{ where }l>0,
\end{equation}
let $\phi$ be a  function from $ {\mathbf N}_{\kappa}^{k}$, such that
\begin{equation}\label{eq:Nev_phi2}
  \phi(z)\ne o(1){\quad\mbox{ as }z\widehat{\rightarrow}\infty}
\end{equation}
and let $\tau=T_{L}[\phi]=l+\phi$.
Then:
\begin{enumerate}
  \item [({ i})] either  $\tau\in  {\mathbf N}_{\kappa}^{k}$ and $\tau(z)^{-1}=o(z)\mbox{ as }z\widehat{\rightarrow}\infty$,
  \item [({ ii})] or  $\tau\in  {\mathbf N}_{\kappa}^{k-1}$.
\end{enumerate}
\end{lemma}
\begin{proof}
  {\bf 1.} {\it Verification of $({ i})$ for $\phi\in  {\mathbf N}_{\kappa}^{k}$, such that $z\phi(z)$ has a GPNT 
  at $\infty$ and$:$}
\begin{equation}\label{eq:3.11x}
   \lim_{z\widehat{\rightarrow}\infty}\phi(z)
   =-\infty.
\end{equation}
The function $z\tau(z)=lz+z\phi(z)$ has GPNT at $\infty$ of the same multiplicity as $z\phi(z)$ and by Theorem~\ref{prop:GPNT}
$z\tau(z)\in  {\mathbf N}_{k}$. Hence $\tau\in  {\mathbf N}_{\kappa}^{k}$ and
the equality
\[
  \lim_{z\widehat{\rightarrow}\infty}\frac{1}{z\tau(z)}=
  \lim_{z\widehat{\rightarrow}\infty}\frac{1}{lz+z\phi(z)}{=0}
\]
implies that  $\tau(z)^{-1}=o(z)$  as $z\widehat{\rightarrow}\infty$. Thus (1) holds.

{\bf 2.} {\it Verification of $({ i})$ for $\phi\in  {\mathbf N}_{\kappa}^{k}$, such that
$
  \lim\limits_{z\widehat{\rightarrow}\infty}\phi(z)=a $ and $a>0$$:$}
In this case  $z\phi(z)$ admits the representation
\begin{equation}\label{eq:3.14A}
 z\phi(z)=az+z\phi_1(z),
 \end{equation}
 where  $\lim\limits_{z\widehat{\rightarrow}\infty}\phi_1(z)=0$ and $\phi_1\in  {\mathbf N}_{\kappa}^{k}$ by Theorem~\ref{prop:GPNT}.
 The function
 $z\tau(z)=lz+z\phi(z)$  has no  GPNT at $\infty$ and therefore $\tau\in  {\mathbf N}_{\kappa}^{k}$. Moreover,
\begin{equation}\label{eq:3.16L}
  \lim_{z\widehat{\rightarrow}\infty}\frac{1}{z\tau(z)}=
  \lim_{z\widehat{\rightarrow}\infty}\frac{1}{(l+a)z+z\phi(z)}=0.
\end{equation}
and hence $({ i})$ holds.

{\bf 3.} {\it If   $\phi\in  {\mathbf N}_{\kappa}^{k}$, $\lim\limits_{z\widehat{\rightarrow}\infty}\phi(z)=a$ and $a<-l${,} then also $({ i})$ holds$:$}
In this case
\[
\kappa_\infty(z\phi)=1 \quad\mbox{and}\quad \kappa_\infty(z\tau)=1,
\]
that is the functions $z\phi(z)$ and $z\tau(z)$ have GPNT at $\infty$ of the same multiplicity.
By Theorem~\ref{prop:GPNT}  $\tau\in  {\mathbf N}_{\kappa}^{k}$ and by~\eqref{eq:3.16L} $\tau(z)^{-1}=o(z)$   as $z\widehat{\rightarrow}\infty$. Hence $({ i})$ holds.

{\bf 4.} {\it  If   $\lim\limits_{z\widehat{\rightarrow}\infty}\phi(z)=a$ and $a\in[-l,0)$ then $({ ii})$ holds$:$}
In this case
\[
\kappa_\infty(z\phi)=1 \quad\mbox{and}\quad \kappa_\infty(z\tau)=0,
\]
and hence $z\tau\in{\mathbf N}_{k-1}$ by Theorem~\ref{prop:GPNT}. Thus $\tau\in  {\mathbf N}_{\kappa}^{k-1}$.
\end{proof}

\begin{corollary}\label{cor:fo0}
  Let a matrix valued function $W_{2n}(z)$ of the form~\eqref{eq:ResM_2N} belong to the class $\cU_0^0(J)$, let
 $\phi\in  {\mathbf N}_{\kappa}^{k}$, \eqref{eq:Nev_phi1} fails to hold, let $n\ge \mbox{min }\{\kappa,k\}$ and let $f=T_{W_{2n}}[\phi]$.
Then there exists $r\in\dZ_+$ $(r\le\mbox{min }\{\kappa,k\})$, such that :
\begin{enumerate}
  \item [({ i})] either  $f\in  \cM_{\kappa-r}^{k-r}(\mathbf{s},2(n-r)-2)$;
  \item [({ ii})] or  $f\in  \cM_{\kappa-r}^{k-r-1}(\mathbf{s},2(n-r)-3)$.
\end{enumerate}
In particular, $f(z)=o(1)$ as $z\widehat{\rightarrow}\infty$.
\end{corollary}
\begin{proof}
By Theorem~\ref{thm:DescrOdd} the matrix valued function $W_{2n}(z)$
admits the factorization
\begin{equation}\label{eq:W+fact}
  W_{2n}(z)=M_1(z)L_1\dots M_{n}(z)L_{n}.
\end{equation}
Let us denote $\tau_1:=T_{L_n}[\phi]$.
By Lemma~\ref{lem:W_l2} there are three possibilities
\begin{enumerate}
  \item [(a1)] either $\tau_1\in  \mathbf{N}_{\kappa}^{k}$ and $\tau_1^{-1}(z)=o(z)$;
  \item[(a2)] or $\tau_1\in  \mathbf{N}_{\kappa}^{k-1}$ and $\tau_1^{-1}(z)=o(z)$;
  \item [(a3)] or $\tau_1\in  \mathbf{N}_{\kappa}^{k-1}$ and $\tau_1^{-1}(z)\ne o(z)$.
\end{enumerate}
In the  case (a1) one obtains by Theorem~\ref{thm:DescrOdd}  $f\in  \cM_{\kappa}^{k}(\mathbf{s},2n-2)$, which gives ({i}) with $r=0$.

In the  case (a2)  by Theorem~\ref{thm:DescrOdd}  $f\in  \cM_{\kappa}^{k-1}(\mathbf{s},2n-2)\subset \cM_{\kappa}^{k-1}(\mathbf{s},2n-3)$,
 which gives ({ ii}) with $r=0$.

In the  case (a3)  we apply Lemma~\ref{lem:W_m2} to the function $\phi_1(z)=T_{M_{n}}[\tau_1(z)]$ and then again there are three possibilities:
\begin{enumerate}
    \item[(b1)] $\phi_1:=T_{M_n}[\tau]\in  \mathbf{N}_{\kappa}^{k-1}$ and $\phi_1(z)=o(1)$;
  \item [(b2)] $\phi_1:=T_{M_n}[\tau]\in  \mathbf{N}_{\kappa-1}^{k-1}$  and $\phi_1(z)=o(1)$;
  \item [(b3)]   $\phi_1:=T_{M_n}[\tau]\in  \mathbf{N}_{\kappa-1}^{k-1}$ and $\phi_1(z)\ne o(1)$.
\end{enumerate}
In the case (b1)  by Theorem~\ref{thm:4.2} one obtains  $f\!\!\in\! \! \cM_{\kappa}^{k\!-\!1}(\mathbf{s},2n\!-\!3)$ which gives (2) with $r\!=\!0$.

In the case (b2)  by Theorem~\ref{thm:4.2}  $f\in  \cM_{\kappa-1}^{k-1}(\mathbf{s},2n-3)\subset\cM_{\kappa-1}^{k-1}(\mathbf{s},2n-4)$ which gives ({ i}) with $r=1$.

In the case (b3) one should continue this process based on Lemma~\ref{lem:W_m2}, Lemma~\ref{lem:W_l2}, Theorem~\ref{thm:4.2} and Theorem~\ref{thm:DescrOdd}.

If, for instance, $k\le\kappa$ and if the process will not stop until the step  $r=k$, then when  applying Lemma~\ref{lem:W_l2} to $\phi_k\in{\mathbf N}_{\kappa-k}^{0}$ one obtains only one possibility: $\tau_{k+1}\in{\mathbf N}_{\kappa-k}^{0}$ and $\tau_{k+1}(z)^{-1}=o(z)$ as $z\widehat{\to}\infty$. Then the process stops and  $f\in  \cM_{\kappa-k}^{0}(\mathbf{s},2(n-k)-2)$   by Theorem~\ref{thm:DescrOdd}.

Similarly, one can treat the case $\kappa<k$ by using Lemma~\ref{lem:W_m2}  and Theorem~\ref{thm:4.2}.
\end{proof}


\section{Full indefinite moment problem $MP_{\kappa}^{k}(\textbf{s} )$.}
\subsection{Indefinite moment problem $MP_{\kappa}^{k}(\textbf{s} )$  for  ${\mathbf s}\in\cH_0^0$.}\label{sec:4.1}

 Recall, that the moment problem $MP_{\kappa}^{k}({\mathbf s})$ is called {\it indeterminate} if it has more then one solution.
A sequence $\textbf{s}$ is called {\it nondegenerate}, if
\begin{equation}\label{eq:MP_nondeg}
     \mbox{there is $N\in\dN$}, \,\mbox{ such that }\,\quad D_n\ne 0, \quad D_n^+\ne 0\quad \mbox{for all \,\, $n\ge N$.}
\end{equation}

In this subsection we consider the indefinite moment problem $MP_{\kappa}^k(\textbf{s})$ for a  nondegenerate sequence ${\mathbf s}\in\cH_0^0$.
As is known, see~\cite[Theorem 0.5]{Akh}, for  a  nondegenerate sequence ${\mathbf s}\in\cH_0^0$ the corresponding moment problem $MP_{0}^0(\textbf{s})$ is indeterminate, if and only if
    \begin{equation}\label{eq:IndS}
        \sum_{i=1}^\infty m_i<\infty\quad\mbox{and}\quad\sum_{i=1}^\infty l_i<\infty.
    \end{equation}
\begin{theorem}\label{thm:Full_MP00}
  Let $\textbf{s}$ be a  nondegenerate sequence from $\cH_0^0$ and let \eqref{eq:IndS} holds.
  Then:
  \begin{enumerate}
    \item [({ i})] The moment problem $MP_{\kappa}^k(\textbf{s})$ is solvable and  indeterminate for any pair of $\kappa,k\in\dZ_+$.
    \item  [({ ii})]  The sequence of resolvent matrices $ W_{2n}(z)$ converges to an entire matrix valued function $W_{\infty}^+(z)=( w_{ij}^+(z))_{i,j=1}^2$ of order $\le$ 1/2.
   \item [({ iii})]  The formula
\begin{equation}\label{eq:MP_descrx}
f(z)=\frac{w_{11}^+(z)\tau(z)+w_{12}^+(z)}
    { w_{21}^+(z)\tau(z)+ w_{22}^+(z)}
\end{equation}
establishes a one-to-one correspondence between the class $\cM_{\kappa}^k(\textbf{s})$ and the set of functions $\tau\in \mathbf{N}_{\kappa}^{k}$.
\end{enumerate}
\end{theorem}
\begin{proof}
{\bf 1.} {\it Verification of ({ ii}):}
Due to~\cite[Appendix~II.13]{KacK68} the Stieltjes moment problem $\mathcal{M}_{0}^{0}({\mathbf {s}^{(N)}})$ is indeterminate if and only if \eqref{eq:IndS} holds
and in this case the sequence $W_{2N}(z)$ (see~\eqref{eq:ResM_2N})
of its resolvent matrices converges to an entire matrix valued function $W_\infty^+(z)=( w_{ij}^{+}(z))_{i,j=1}^2$ of order $\le$ 1/2.

{\bf 2.} {\it Verification of the implication $f\in\cM_\kappa^k(\textbf{s})\Longrightarrow \tau\in  \mathbf{N}_{\kappa}^{k}$ {$:$}} 
Since $f\in\cM_\kappa^k(\textbf{s})$ then $f\in\cM_\kappa^k(\textbf{s},2n-1)$ for every $n\in\dN$.
By Theorem~\ref{thm:4.2} there exists a sequence of functions $\phi_N(z)\in \mathbf{N}_{\kappa}^{k}$ such that \eqref{eq:Nev_phi1} holds and
\begin{equation}\label{eq:fTW}
  f(z)=T_{W_{2N}(z)}[\phi_N(z)]\quad\textup{for every $N\in\dN$.}
\end{equation}
It follows from~\eqref{eq:fTW} that
{
\[
-(Q^{+}_{2N-1}(z)-f(z)P^{+}_{2N-1}(z))\phi_N(z)=Q^{+}_{2N}(z)-f(z)P^{+}_{2N}(z).
\]
Notice that $Q^{+}_{2N-1}(z)-f(z)P^{+}_{2N-1}(z)\not\equiv 0$, since otherwise we would have
$Q^{+}_{2N}(z)-f(z)P^{+}_{2N}(z)\equiv 0$, which contradicts the
 generalized Liouville-Ostrogradsky identity \eqref{6p.eq:gLN1}. 
}

As was mentioned above the matrix valued functions $W_{{2N}}(z)$ converges locally uniformly in $\dC$ to  $W_\infty^+(z)=(w_{ij}^+(z))_{i,j=1}^2$, in particular,
\[
w_{11}^+(z)=\lim_{N\to\infty}Q^{+}_{{2N}-1}(z),\quad
w_{21}^+(z)=\lim_{N\to\infty}P^{+}_{{2N}-1}(z).
\]
By Corollary~\ref{cor:AC} $\frac{Q^{+}_{{2N}-1}}{P^{+}_{{2N}-1}}\in \mathbf{N}_{0}^{0}=S$. Hence the limiting function $\frac{w_{11}^+(z)}{w_{21}^+(z)}$ belongs to $\mathbf{N}_{0}^{0}=S$ and it corresponds to $\phi=\infty$ in the linear fractional transformation $f=T_{W_\infty^+}[\tau]$. Apart from this case  the function
\begin{equation}\label{eq:DenomTau}
  w_{11}^+(z)-f(z)w_{21}^+(z)=\lim_{N\to\infty}(Q^{+}_{{2N}-1}(z)-f(z)P^{+}_{{2N}-1}(z))
\end{equation}
is not identically equal to 0. Let $\Omega$ be the open set of points in $\dC_+$ such that $w_{11}^+(z)-f(z)w_{21}^+(z)\ne 0$. Then for every point $z\in\Omega$ the sequence of functions
{
\[
\phi_N(z):=\frac{Q^{+}_{{2N}}(z)-f(z)P^{+}_{{2N}}(z)}{Q^{+}_{{2N}-1}(z)-f(z)P^{+}_{{2N}-1}(z)}
\]}
is correctly defined in a neighborhood of $z$ and converges locally uniformly in $\Omega$ to a function $\phi(z)$. Since $\phi_N\in \mathbf{N}_{\kappa}^{k}$ then $\phi\in \mathbf{N}_{\kappa'}^{k'}$ with $\kappa'\le\kappa$ and $k'\le k$.
It follows from~\eqref{eq:fTW} that
\begin{equation}\label{eq:fTWx}
  f(z)=T_{W_\infty^+(z)}[\phi(z)]
\end{equation}
and by Lemma~\ref{lem:W} $f\in \mathbf{N}_{\kappa''}^{k''}$ with $\kappa''\le\kappa'\le\kappa$ and $k''\le k'\le k$.
Since $f\in \mathbf{N}_{\kappa}^{k}$ this implies $\kappa''=\kappa'=\kappa$ and $k''= k'= k$ and hence $\tau\in \mathbf{N}_{\kappa}^{k}$.

{\bf 3.} {\it
{ Proof
of the fact that}  $T_{W_\infty^+}[\tau]$ satisfies \eqref{eq:Nev_phi1} for every $\tau\in \mathbf{N}_\kappa^k$, $\kappa,k\in\dZ_+$.}
Application of the Schur algorithm to the Stieltjes moment problem $\mathcal{M}_{\kappa}^{k}({\mathbf s})$
gives on the $N-$th  step an induced sequence ${\mathbf {s}^{(N)}}$, which  can be found as the sequence of  coefficients of the series expansion $-{\sum\limits_{i=0}^\infty{s_i^{(N)}}{z^{-(i+1)}}}$ corresponding to the continued fraction
 \eqref{s8.0}.
Denote by $W^{(N)}_{2j}(z)$ the resolvent matrix of the corresponding Stieltjes moment problem $\mathcal{M}_{0}^{0}({\mathbf {s}^{(N)}},2j-1)$.
\begin{equation}\label{eq:ResNN}
  W^{(N)}_{2j}(z)=
      \begin{pmatrix}
        Q^{N,+}_{2j-1}(z) & Q^{N,+}_{2j}(z) \\
        P^{N,+}_{2j-1}(z) & P^{N,+}_{2j}(z) \\
    \end{pmatrix},\quad j\in\dN.
\end{equation}
Then for every $j\in\dN$ the resolvent matrix $W_{2j}(z)$ admits the factorization (see~\cite[Proposition 4.7]{DK17})
\begin{equation}\label{eq:Res_jN}
W_{2j}(z)=W_{{2N}}(z)W_{2(j-N)}^{(N)}(z).
\end{equation}
Taking the limit as $j\to\infty$ one obtains the following factorization of the entire matrix valued function $W_\infty^+(z)$
\begin{equation}\label{eq:Winf_fact}
  W_\infty^+(z)=W_{{2N}}(z)W_\infty^{(N)}(z),
\end{equation}
where $W_\infty^{(N)}(z)$ is the resolvent matrix of the induced Stieltjes moment problem $\mathcal{M}_{0}^{0}({\mathbf {s}^{(N)}})$.
Then $f=T_{W_\infty^+}[\tau]$ admits the reresentation
\begin{equation}\label{eq:f_rep}
  f=T_{W_{{2N}}}[\phi_N],\quad\mbox{where}\quad \phi_N=T_{W_\infty^{(N)}}[\tau]\in \mathbf{N}_{\kappa'}^{k'},\,\,
\mbox{with $\kappa'\le \kappa$ and $k'\le k$}.
\end{equation}
Since ${W_{{2N}}}\in\cU_0^0$  by Corollary~\ref{cor:fo0} one obtains $f=o(1)$, if $N>\mbox{min }\{\kappa,k\}$.

{\bf 4.} {\it
Verification of the implication $ \tau\in \mathbf{N}_{\kappa}^{k}\Longrightarrow f=T_{W_\infty^+}[\tau]\in\cM_\kappa^k(\textbf{s})$:}
Due to the above item~({ iii}) the function $\phi_N$ from \eqref{eq:f_rep} satisfies the condition~\eqref{eq:Nev_phi1}.

Next By Theorem~\ref{thm:4.2} $f\in\cM_{\kappa'}^{k'}(\mathbf{s},2N-1)$. Since $N$ can be chosen arbitrarily large  $f\in\cM_{\kappa'}^{k'}(\mathbf{s})$. Now it follows from item $\mathbf{2}$ that $\tau\in \mathbf{N}_{\kappa'}^{k'}$ and hence $\kappa'=\kappa$ and $k'=k$.

{\bf 5.} {\it Verification of ({ i}):} Solvability of the indefinite Stieltjes moment problem $MP_{\kappa}^k(\textbf{s})$  for any pair of $\kappa,k\in\dZ_+$ follows from item $\mathbf{4}$. 
The formula
\eqref{eq:MP_descrx}  gives two different solutions of the problem  $MP_{\kappa}^k(\textbf{s})$ for different parameter functions $\tau_1,\tau_2\in \mathbf{N}_\kappa^k$ and thus the problem $MP_{\kappa}^k(\textbf{s})$ is indeterminate.
\end{proof}



\subsection{Indefinite moment problem $MP_{\kappa}^{k}(\textbf{s} )$, general case.}


\begin{theorem}\label{thm:Full_MP}
  Let $\textbf{s}$ be a  nondegenerate sequence from $\cH_{\kappa_0}^{k_0,reg}$, $\kappa_0,k_0\in\dN$ and  let ${l}_j$ and $m_j(z)$ ($j\in\mathbb{N}$) be parameters  of the generalized $\mathbf{S}-$fraction (\ref{s4.2}). Then the moment problem $MP_{\kappa_0}^{k_0}(\textbf{s})$  is indeterminate, if and only if
      \begin{equation}\label{eq:ldlsa1}
        \sum_{i=1}^\infty m_i(0)<\infty\quad\mbox{and}\quad\sum_{i=1}^\infty l_i<\infty.
    \end{equation}
If \eqref{eq:ldlsa1} holds, then:
\begin{enumerate}
  \item [(i)] The sequence of resolvent matrices $ W_{2n}(z)$ converges to an entire matrix valued function $W_{\infty}^+(z)=( w_{ij}^+(z))_{i,j=1}^2$ of order $\le$ 1/2.
   \item [(ii)] The moment problem $MP_{\kappa}^k(\textbf{s})$  is solvable, 
   if and only if
\begin{equation}\label{eq:MPkk_Solv}
  \kappa_0\le\kappa,\quad\mbox{and}\quad k_0\le k.
\end{equation}
   \item [(iii)]  The formula
\begin{equation}\label{eq:MP_descrxA}
f(z)=\frac{w_{11}^+(z)\tau(z)+w_{12}^+(z)}
    { w_{21}^+(z)\tau(z)+ w_{22}^+(z)}
\end{equation}
establishes a one-to-one correspondence between the class {$\cM_{\kappa}^k(\textbf{s})$} and the set of functions $\tau\in \mathbf{N}_{\kappa-\kappa_0}^{k-k_0}$.
\end{enumerate}
\end{theorem}
\begin{proof}
{\bf 1.} {\it Redaction of the indefinite moment problem $MP_{\kappa}^k(\textbf{s})$ to a classical one:}
Let us choose $N$ big enough, so that
\begin{equation}\label{choiceN}
  \nu_-(S_j)=\kappa_0=\nu_-(S_{n_N}),\quad
\nu_-(S_j^{(1)})=k_0=\nu_-(S_{n_N}^{(1)})\quad\mbox{for all}\quad j\ge n_N.
\end{equation}
 Then the induced sequence ${\mathbf {s}^{(N)}}$ which arises on the $N-$th step of the Schur algorithm (see~Section~\ref{sec:2.2}) belongs to the class $\cH_0^0$.
The corresponding Stieltjes moment problem $\mathcal{M}_{0}^{0}({\mathbf {s}^{(N)}})$ is classical.
By Lemma~\ref{lem:2.4}
\begin{equation}\label{eq:InducedMP_2}
  f\in\cM_\kappa^k(\mathbf{s}) \Longleftrightarrow  f_N=T_{W_{2n_N}}[f]\in\cM_{\kappa-\kappa_0}^{k-k_0}(\mathbf{s}^{(N)}).
\end{equation}
In particular,
\begin{equation}\label{eq:InducedMP_2A}
  f\in\cM_{\kappa_0}^{k_0}(\mathbf{s}) \Longleftrightarrow  { f_N=T_{W_{2n_N}}[f]\in\cM_{0}^{0}(\mathbf{s}^{(N)}).}
\end{equation}
Hence  the problem $\mathcal{M}_{\kappa_0}^{k_0}({\mathbf {s}^{(N)}})$ is indeterminate if and only if  the Stieltjes moment problem $\mathcal{M}_{0}^{0}({\mathbf {s}^{(N)}})$ is indeterminate. The latter is equivalent to~\eqref{eq:ldlsa1}.

{\bf 2.} {\it Verification of (i):}
Let $P^{N,+}_{j}(z)$ and $Q^{N,+}_{j}(z)$ be generalized  Stieltjes polynomials associated with the moment problem $\mathcal{M}_{0}^{0}({\mathbf {s}^{(N)}})$ and let $W^{(N)}_{2j}$ (see~\eqref{eq:ResNN}).
If ~\eqref{eq:ldlsa1} holds, then
the Stieltjes moment problem $\mathcal{M}_{0}^{0}({\mathbf {s}^{(N)}})$ is indeterminate
and the sequence
$  W^{(N)}_{2j}$ converges to an entire matrix valued function $W^{(N)}_\infty(z)=( w_{ij}^{(N)}(z))_{i,j=1}^2$ of order $\le$ 1/2.

The resolvent matrices $W_{2j}(z)$ of the indefinite moment problem $MP_{\kappa_0}^{k_0}(\textbf{s}, 2j-1)$ are connected with the resolvent matrices $W^{(N)}_{2(j-N)}(z)$ of the induced moment problem $MP_{0}^{0}(\textbf{s}^{(N)}, 2(j-N)-1)$  by the formula~\eqref{eq:Res_jN}.
Therefore, the sequence of matrix valued functions $W_{2n-1}(z)$ also converges to an entire matrix valued function $W_\infty^+(z)=( w_{ij}^+(z))_{i,j=1}^2$ of order $\le$ 1/2, which is connected with $W_\infty^{(N)}(z)$  by the formula~\eqref{eq:Winf_fact}.
%

{\bf 3.} {\it Verification of (ii) and (iii) :}
By Proposition~\ref{prop:2.1} (v) the inequalities~\eqref{eq:MPkk_Solv} are necessary fo solvability of the problem $\cM_\kappa^k(\textbf{s})$.

Now assume that \eqref{eq:MPkk_Solv} holds, let  $\tau\in \mathbf{N}_{\kappa-\kappa_0}^{k-k_0}$ and let $f=T_{W_\infty^+}[\tau]$. Then by Theorem~\ref{thm:Full_MP00}
\begin{equation}\label{eq:f_N}
  f_N=T_{W_\infty^{(N)}}[\tau]\in\cM_{\kappa-\kappa_0}^{k-k_0}(\mathbf{s}^{(N)}).
\end{equation}
By \eqref{eq:Winf_fact} and Lemma~\ref{lem:2.4}
\begin{equation}\label{eq:ff_N}
  f=T_{W_{2N}}[f_N]\in\cM_{\kappa}^{k}(\mathbf{s}).
\end{equation}

Conversely, let $f\in\cM_{\kappa}^{k}(\mathbf{s})$. Then by  Lemma~\ref{lem:2.4} there is a function $f_N\in\cM_{\kappa-\kappa_0}^{k-k_0}(\mathbf{s}^{(N)})$ such that $ f=T_{W_{2N}}[f_N]$. Hence by Theorem~\ref{thm:Full_MP00} there exists a function $\tau\in \mathbf{N}_{\kappa-\kappa_0}^{k-k_0}$, such that $ f_N=T_{W_\infty^{(N)}}[\tau]$. Therefore,
$ f=T_{W_\infty^{+}}[\tau]$ for $\tau\in \mathbf{N}_{\kappa-\kappa_0}^{k-k_0}$.
\end{proof}


\begin{proposition}\label{prop:def}
  Let $\textbf{s}=\left\{s_{j}\right\}_{j=0}^{ \infty}\in \cH_\kappa^{k,reg}$, let $P_i(z)$  and $Q_i(z)$ ($i\in\dZ_+$)  be Lanczos polynomials of the first and second kind and let ${l}_j$ and $m_j(z)$ ($j\in\mathbb{N}$) be parameters  of the generalized $\mathbf{S}-$fraction (\ref{s4.2}). Assume additionally that
  \begin{equation}\label{eq:l>0}
    l_j>0\quad\mbox{for all\,\, $j\in\dN$}.
  \end{equation}
  Then the following statements are equivalent:
  \begin{enumerate}
    \item [(i)] the moment problem $MP_{\kappa}(\textbf{s})$  is indeterminate;
    \item [(ii)]  the following series converge
    \begin{equation}\label{eq:M<infty}
    \sum_{i=0}^\infty |P_i(0)|^2\wt b_i^{-1}<\infty\quad\mbox{ and }\quad\sum_{i=0}^\infty |Q_i(0)|^2\wt b_i^{-1}<\infty;\end{equation}
    \item [(iii)]  the following series converges
  \begin{equation}\label{eq:M<infty.bf}
 \sum_{i=1}^\infty (l_1+l_2+\cdots+l_i)^2d_{i+1}(0)<\infty.
    \end{equation}
  \end{enumerate}
\end{proposition}
\begin{proof}
  By \cite{DD07,KL79} $(i)$ and $(ii)$ are equivalent.

  Let us show, that $(ii)$ and $(iii)$ are equivalent. By the first equality in \eqref{6p.eq:cor.l1x}
  \[
    Q_i(0)=-(l_1+l_2+\cdots+l_i)P_i(0).
\]
and hence by \eqref{6p.eq:cor.l1}
\[
\sum_{i=0}^N |Q_i(0)|^2\wt b_i^{-1}=\sum_{i=0}^N (l_1+l_2+\cdots+l_i)^2 |P_i(0)|^2\wt b_i^{-1}=
 \sum_{i=1}^N (l_1+l_2+\cdots+l_i)^2d_{i+1}(0).
\]
This proves the implication $(ii)\Rightarrow (iii)$.

Conversely, if {$\sum\limits_{i=0}^\infty |Q_i(0)|^2\wt b_i^{-1}<\infty$}, then the convergence of the series $\sum\limits_{i=0}^\infty |P_i(0)|^2\wt b_i^{-1}$ follows from the inequality
\[
l_1^2\sum_{i=0}^N  |P_i(0)|^2\wt b_i^{-1}\le \sum_{i=0}^N (l_1+l_2+\cdots+l_i)^2 |P_i(0)|^2\wt b_i^{-1}
{=\sum\limits_{i=0}^\infty |Q_i(0)|^2\wt b_i^{-1}}.
\]
This proves the implication $(iii)\Rightarrow (ii)$.
\end{proof}
\begin{remark}
  In the case $\kappa=k=0$ the above criterion for the Hamburger moment problem to be indeterminate is well known, see~\cite[Appendix, Theorem 0.5]{Akh}. When treating system~\eqref{s4.3} as a Stieltjes string with masses $m_j$ and lengthes $l_j$ one can consider series \eqref{eq:M<infty.bf} as "the moment of inertia" of the string.
\end{remark}

\section{Pad\'{e} approximants}\label{sec:5}
\begin{definition}\label{def:Pade}(\cite{BaGr86})
  The $[L/M]$ Pad\'{e} approximant for a formal power series
\begin{equation}\label{eq:FPS}
  -\sum_{j=0}^\infty \frac{s_j}{z^{j+1}}
\end{equation}
 is a ratio
\begin{equation}\label{eq:Pade_con1}
f^{[L/M]}({z})=\frac{A^{[L/M]}(1/{z})}{B^{[L/M]}(1/{z})}
\end{equation}
of polynomials $A^{[L/M]}$, $B^{[L/M]}$ of formal degree $L$, $M$,
respectively, such that $B^{[L/M]}(0)\ne 0$ and
\begin{equation}\label{eq:Pade_con}
f^{[L/M]}(z)+\sum_{j=0}^{L+M} \frac{s_j}{z^{j+1}}
=O\left(\frac{1}{z^{L+M+1}}\right),\quad z\wh{\to} \infty.
\end{equation}
\end{definition}

The $[n/n]$ Pad\'{e} approximant is called diagonal and the $[n/n-1]$ Pad\'{e} approximant  is called subdiagonal.
\begin{remark}\label{rem:Pade}
Notice, that  for diagonal Pad\'{e} approximants the representation~\eqref{eq:Pade_con1} is equivalent to the representation
\[
f^{[n/n]}({z})=\frac{z^nA^{[n/n]}(1/{z})}{z^nB^{[n/n]}(1/{z})}
\]
as a ratio, where the numerator $z^nA^{[n/n]}(1/{z})$ is a polynomial of formal degree $n$ and the denominator $z^nB^{[n/n]}(1/{z})$ is a polynomial of exact degree $n$.
For subdiagonal Pad\'{e} approximants the representation~\eqref{eq:Pade_con1} is equivalent to the representation
\[
f^{[n/n-1]}({z})=\frac{z^nA^{[n/n-1]}(1/{z})}{z^nB^{[n/n-1]}(1/{z})}
\]
as a ratio, where the numerator is a polynomial of formal degree $n$ and the denominator is a polynomial of exact degree $n$ vanishing at $0$.
\end{remark}
Explicit formula for diagonal Pad\'{e} approximants for sequences $\textbf{s}=\left\{s_{j}\right\}_{j=0}^{ \infty}\in\cH_{\kappa}$ was found in~\cite{DD04}, in the classical case $\textbf{s}\in\cH_{0}$  see~\cite{BaGr86,Sim98}.
In this section we will formulate the corresponding statements for sequences $\textbf{s}=\left\{s_{j}\right\}_{j=0}^{ \infty}\in\cH_{\kappa}^{k,reg}$.
\begin{proposition}\label{prop:Diag_Pade}
Let $\textbf{s}=\left\{s_{j}\right\}_{j=0}^{ \infty}\in\cH_{\kappa}^{k,reg}$, $\kappa,k\in\dZ_+$.
Then the $[n/n]$ Pad\'{e} approximant for a formal power series~\eqref{eq:FPS} exists if $n\in\cN({\bf s})$ and
\begin{equation}\label{eq:Pade_con2}
f^{[n_j/n_j]}(z)=-\frac{Q_{j}(z)}{P_{j}(z)}=\frac{Q_{2j}^+(z)}{P_{2j}^+(z)},\quad j\in\dN.
\end{equation}
\end{proposition}
\begin{proof}
We present a proof of this statement from~\cite{DK17} for the convenience of the reader.
It follows from~\eqref{2p.new8.r7} and Theorem~\ref{thm:DescrOdd}  that the function
\[
-\frac{Q_{j}(z)}{P_{j}(z)}=\frac{Q_{2j}^+(z)}{P_{2j}^+(z)}=T_{W_{2j}(z)}[0]
\]
belongs to $\cM({\bf s},2n_j-1)$. Therefore,  the function $-\frac{Q_{j}(z)}{P_{j}(z)}$ has the asymptotic
\begin{equation}\label{eq:asymQP}
-\frac{Q_{j}(z)}{P_{j}(z)}=-\frac{s_0}{z}-\dots-\frac{s_{2n_j-1}}{z^{2n_j}}+O\left(\frac{1}{z^{2n_j+1}}\right),\quad z\wh{\to} \infty.
\end{equation}
{ Since $Q_j(z)$ is a polynomial of degree $n_j-n_1<n_j$ and $P_j(z)$ is a polynomial of exact degree $n_j$
the function
$-\frac{Q_{j}(z)}{P_{j}(z)}$ is the $[n_j/n_j]$ Pad\'{e} approximant for the formal power series~\eqref{eq:FPS} due to Remark~\ref{rem:Pade} and~\eqref{eq:asymQP}.}
\end{proof}

In the following proposition stated in~\cite{DK16} without proof it is shown that the sub-diagonal Pad\'{e} approximants can be calculated in terms of generalized Stieltjes polynomials.
\begin{proposition}\label{prop:sub_Diag_Pade}
Let $\textbf{s}=\left\{s_{j}\right\}_{j=0}^{ \infty}\in\cH_{\kappa}^{k,reg}$, $\kappa,k\in\dZ_+$.
Then the $[n_{j}/n_{j}-1]$  Pad\'{e} approximants for the formal power series~\eqref{eq:FPS} exists and has the form
\begin{equation}\label{eq:Pade_con3}
f^{[n_{j}/n_{j}-1]}(z)=\frac{Q_{2j-1}^+(z)}{P_{2j-1}^+(z)},\quad j\in\dN.
\end{equation}
\end{proposition}
\begin{proof}
It follows from~\eqref{2p.new8.r7}   that
\[
\frac{Q_{2j-1}^+(z)}{P_{2j-1}^+(z)}=T_{W_{2j-1}(z)}[\infty].
\]
By Theorem~\ref{thm:DescrOdd}   the function $\frac{Q_{2j-1}^+(z)}{P_{2j-1}^+(z)}$ belongs to $\cM({\bf s},2n_{j}-2)$, and hence it has the asymptotic
\begin{equation}\label{eq:asymQPsub}
\frac{Q_{2j-1}^+(z)}{P_{2j-1}^+(z)}=-\frac{s_0}{z}-\dots-\frac{s_{2n_{j}-2}}{z^{2n_{j}-1}}+ O\left(\frac{1}{z^{2n_j+1}}\right),\quad z\wh{\to} \infty.
\end{equation}
{ Since $Q_{2j-1}^+(z)$ is a polynomial of degree $n_j-n_1<n_j$ and $P_{2j-1}^+(z)$ is a polynomial of exact degree $n_j$ vanishing at $0$
the function
$-\frac{Q_{2j-1}^+(z)}{P_{2j-1}^+(z)}$ is the $[n_j/n_j-1]$ Pad\'{e} approximant for the formal power series~\eqref{eq:FPS} due to Remark~\ref{rem:Pade} and~\eqref{eq:asymQPsub}.
}\end{proof}
{
\begin{lemma}\label{lem:Pade}
Let $\textbf{s}=\left\{s_{j}\right\}_{j=0}^{ \infty}\in\cH_{\kappa}^{k,reg}$, $\kappa,k\in\dZ_+$,   let $W_{2N}(z)$ be given by~\eqref{eq:ResM_2N},  $N\in\dN$ and let $\mathbf{s}^{(N)}$ be the induced sequence defined in Lemma~\ref{lem:2.4}, and let
\begin{equation}\label{eq:Ser2}
  -\sum_{j=0}^\infty {s_j^{(N)}}{z^{-(j+1)}}
\end{equation}
be the corresponding  formal power series. Then:
\begin{enumerate}
  \item [(i)]  the diagonal  Pad\'{e} approximants for the formal power series~\eqref{eq:FPS}
are connected with diagonal $g^{[n/n]}$  Pad\'{e} approximants for the power series~\eqref{eq:Ser2}
 by the formula
\begin{equation}\label{eq:Pade_N}
  f^{[n_{j}/n_{j}]}(z)=T_{W_{2N}(z)}[g^{[n_{j}-n_N/n_{j}-n_N]}(z)],\quad j>N,\, j\in\dN.
\end{equation}
  \item[(ii)] the subdiagonal Pad\'{e} approximants for the power series~\eqref{eq:FPS}
are connected with subdiagonal $g^{[n/n-1]}$ Pad\'{e} approximants for the power series
\eqref{eq:Ser2}
 by
\begin{equation}\label{eq:Pade_Nsub}
  f^{[n_{j}/n_{j}-1]}(z)=T_{W_{2N}(z)}[g^{[n_{j}-n_N/n_{j}-n_N-1]}(z)],\quad j>N,\, j\in\dN.
\end{equation}
\end{enumerate}
\end{lemma}
\begin{proof}
Consider the induced moment problem $MP_{\kappa-\kappa_N}^{k-k_N^+}(\mathbf{s}^{(N)}, 2(n_j-n_N)-1)$ and let
\begin{equation}\label{eq:ResM_2j-N}
  W_{2(j-N)}^{(N)}(z)=\left(
                \begin{array}{cc}
                  Q^{N,+}_{2(j-N)-1}(z) & Q^{N,+}_{2(j-N)}(z) \\
                  P^{N,+}_{2(j-N)-1}(z) & P^{N,+}_{2(j-N)}(z) \\
                \end{array}
              \right)
\end{equation}
be the resolvent matrix of this moment problem. Then the matrices $  W_{2j}(z)$ and $W_{2(j-N)}^{(N)}(z)$ are connected by (see~\eqref{eq:Res_jN})
\begin{equation}\label{eq:W_fact_N}
  W_{2j}(z)=W_{2N}(z)W_{2(j-N)}^{(N)}(z).
\end{equation}
Similarly,  the resolvent matrix
\begin{equation}\label{eq:ResM_2j-N-1}
    W_{2(j-N)-1}^{(N)}(z)=\left(
                \begin{array}{cc}
                  Q^{N,+}_{2(j-N)-1}(z) & Q^{N,+}_{2(j-N)-2}(z) \\
                  P^{N,+}_{2(j-N)-1}(z) & P^{N,+}_{2(j-N)-2}(z) \\
                \end{array}
              \right)
\end{equation}
 of the moment problem $MP_{\kappa-\kappa_N}^{k-k_N}(\mathbf{s}^{(N)}, 2(n_j-n_N)-1)$ is connected with
the matrix $  W_{2j-1}(z)$ by
\begin{equation}\label{eq:W_fact_N-1}
  W_{2j-1}(z)=W_{2N}(z)W_{2(j-N)-1}^{(N)}(z).
\end{equation}

By Proposition~\ref{prop:Diag_Pade}  diagonal  Pad\'{e} approximants $g^{[n_{j}-n_N/n_{j}-n_N]}(z)$ for the formal power series~\eqref{eq:Ser2} are given by
\begin{equation}\label{eq:Pade_con2N}
g^{[n_{j}-n_N/n_{j}-n_N]}(z)=\frac{Q_{2(n_j-n_N)}^{N,+}(z)}{P_{2(n_j-n_N)}^{N,+}(z)}
=T_{W_{2(j-N)}^{(N)}(z)}[0],\quad j\in\dN.
\end{equation}
It follows from the factorization  formula~\eqref{eq:W_fact_N}
and \eqref{eq:Pade_con2N} that
\[
T_{W_{2N}(z)}[g^{[n_{j}-n_N/n_{j}-n_N]}(z)]=T_{W_{2j}(z)}[0].
\]
Hence by Proposition~\ref{prop:Diag_Pade} $T_{W_{2N}(z)}[g^{[n_{j}-n_N/n_{j}-n_N]}(z)]$ coincides with the diagonal  Pad\'{e} approximants $f^{[n_{j}/n_{j}]}(z)$ for the formal power series~\eqref{eq:FPS}.

By Proposition~\ref{prop:sub_Diag_Pade}  subdiagonal  Pad\'{e} approximants $g^{[n_{j}-n_N/n_{j}-n_N-1]}(z)$ for the formal power series~\eqref{eq:Ser2} are given by
\begin{equation}\label{eq:Pade_con2Nsub}
g^{[n_{j}-n_N/n_{j}-n_N-1]}(z)=\frac{Q_{2(n_j-n_N)-1}^{N,+}(z)}{P_{2(n_j-n_N)-1}^{N,+}(z)}
=T_{W_{2(j-N)-1}^{(N)}(z)}[\infty],\quad j\in\dN.
\end{equation}
It follows from the factorization  formula~\eqref{eq:W_fact_N-1}
and \eqref{eq:Pade_con2Nsub} that
\[
T_{W_{2N}(z)}[g^{[n_{j}-n_N/n_{j}-n_N]}(z)]=T_{W_{2j}(z)}[\infty].
\]
Hence by Proposition~\ref{prop:sub_Diag_Pade} $T_{W_{2N}(z)}[g^{[n_{j}-n_N/n_{j}-n_N-1]}(z)]$ coincides with the subdiagonal  Pad\'{e} approximants $f^{[n_{j}/n_{j}-1]}(z)$ for the formal power series~\eqref{eq:FPS}.
%
\end{proof}
}
\begin{theorem}\label{thm:Pade}
  Let $\textbf{s}=\left\{s_{j}\right\}_{j=0}^{ \infty}\in\cH_{\kappa}^{k,reg}$, $\kappa,k\in\dZ_+$. Then:
  \begin{enumerate}
    \item [(i)] If the problem $MP_\kappa^k(\mathbf{s})$ is determinate, then diagonal and  subdiagonal Pad\'{e} approximants converge to the unique solution of $MP_\kappa^k(\mathbf{s})$ locally uniformly on $\dC\setminus\dR_+$.
    \item [(ii)] If the problem $MP_\kappa^k(\mathbf{s})$ is indeterminate, then the sequence $f^{[n/n-1]}$ of subdiagonal Pad\'{e} approximants converges  locally uniformly on $\dC\setminus\dR_+$, while the sequence $f^{[n/n]}$ of diagonal Pad\'{e} approximants is not convergent but precompact in the topology of locally uniform convergence.
  \end{enumerate}
\end{theorem}
\begin{proof}
  Let us choose $N$ big enough, so that
\begin{equation}\label{choiceNN}
  \nu_-(S_j)=\kappa=\nu_-(S_{n_N}),\quad
\nu_-(S_j^{(1)})=k=\nu_-(S_{n_N}^{(1)})\quad\mbox{for all}\quad j\ge n_N.
\end{equation}
 Then the induced sequence ${\mathbf {s}^{(N)}}$  belongs to the class $\cH_0^0$ and by Lemma~\ref{lem:Pade} the problem of convergence of diagonal and  subdiagonal Pad\'{e} approximants is reduced to the corresponding problem for  diagonal and  subdiagonal Pad\'{e} approximants for the series $-\sum\limits_{j=0}^\infty {s_j^{(N)}}{z^{-(j+1)}}$ corresponding to the classical Stieltjes moment problem $\mathcal{M}_{0}^{0}({\mathbf {s}^{(N)}})$. The stated results for  classical Stieltjes moment problems were proved in~\cite[Theorems~5.30, 5.31]{Sim98}.
\end{proof}
\section{Example. Laguerre polynomials}\label{sec:6}
 { The monic} Laguerre polynomials $\widetilde{L}_{n}(z,\alpha):=(-1)^nz^{-\alpha}e^z(z^{\alpha+n}e^{-z})^{(n)}$  are solutions of the three--term difference equation (see \cite{seg})
\begin{equation}\label{ex.l.1}\begin{split}
  y_{n+1}(z)+(2n+\alpha+1-z)y_{n}(z)+(n+\alpha)ny_{n-1}(z),\quad n\in\mathbb{N}.
\end{split}\end{equation}
subject to the initial conditions $    \widetilde{L}_{-1}(z,\alpha)\equiv0$, $\widetilde{L}_{0}(z,\alpha)\equiv1$.

If $\alpha>-1$, then the polynomials
$\{\widetilde{L}_{n}(z,\alpha)\}_{n=0}^{\infty}$ are orthogonal  in the Hilbert space $L_2(\mathbb{R}_+,w_\alpha)$, with the weight function  $w_\alpha(z)=z^\alpha e^{-z}$.

Here we consider the case when $\alpha<-1$ and $\alpha$ is not a negative integer. The case when  $\alpha$ is a negative integer was treated in~\cite{LSh92}.
If $-k-1<\alpha<-k$, $k\in\dN$, then $\widetilde{L}_{n}(z,\alpha)$  are orthogonal polynomials  with respect to the indefinite inner product  (see \cite{D98}, \cite{MK78})
\begin{equation}\label{introduction1}\left\langle f,g\right\rangle_{\alpha}=\int\limits_{0}^{\infty}x^{\alpha}\left(e^{-x}f\overline{g}-\sum_{j=0}^{k-1}
(e^{-x}f\overline{g})^{(j)}(0)\frac{x^{j}}{j!}\right)dx.\end{equation}

Polynomials $\widetilde{Q}_{n}(x,\alpha)$ of the second kind are defined as solutions of~\eqref{ex.l.1} subject to the initial conditions $  \widetilde{Q}_{-1}(z,\alpha)\equiv-1$ and  $\widetilde{Q}_{0}(z,\alpha)\equiv0$.

The  $P$ -- fraction corresponding to the  system \eqref{ex.l.1} has the form~\eqref{eq:Pfrac}  with the atoms $(a_n,b_n)$ given by
\begin{equation}\label{ex.l.3''}
    b_0=\Gamma(1+\alpha),\quad b_n=n(n+\alpha)\quad\mbox{and}\quad a_{n-1}(z)= z-2n-\alpha+1,\quad n\in\mathbb{N}.
\end{equation}

The moments $s_n$  for all $n\in\dN$ are defined by
\begin{equation}\label{ex.lem.1.1.2.1}
    s_n=\mathfrak{S}(z^n)=\Gamma(n+\alpha+1).
\end{equation}
By the Buslaev formula, see~{\cite[formula (13)]{Bus10}} the determinants of $S_n$  and $S_n^+$ take the form
\[
D_n=b_0^nb_1^{n-1}\dots b_{n-1}=\prod_{j=1}^n(j-2)!\Gamma(\alpha+j), \quad
D_n^+=\prod_{j=1}^n(j-2)!\Gamma(\alpha+j+1),\quad n\in\dN.
\]
Therefore, the sequence $\textbf{s}=\{s_n\}_{n=0}^\infty$ is regular and the set of its normal indices coincides with $\dN$ and 
\[
\textbf{s}\in\mathcal{H}_{k}^{k-1}, \,\mbox{ if }\,  -2k<\alpha<-2k+1,\quad
\textbf{s}\in\mathcal{H}_{k}^{k}, \,\mbox{ if }\,  -2k-1<\alpha<-2k;
\]

Next, it follows from \eqref{int5x} that the parameters of the generalized S-fraction take the form

  \[
 l_n:=\frac{D_n^2}{D_n^+D_{n-1}^+}=\frac{(n-1)!\Gamma^2(1+\alpha)}{\Gamma(1+\alpha+n)},\, m_n:=\frac{(D_{n-1}^+)^2}{D_nD_{n-1}}=\frac{\Gamma(\alpha+n)}{(n-1)!\Gamma^2(1+\alpha)}\quad (n\in\mathbb{N}).
  \]

The monic Laguerre polynomials $\widetilde{L}_{n}(z,\alpha)$ can be calculated by (see \cite{seg},\cite{suetin})
\begin{equation}\label{ex.l.3}
    \widetilde{L}_{n}(z,\alpha)
    =\sum_{k=0}^{n}\binom{n}{k}\frac{\Gamma(n+\alpha+1)}{\Gamma(k+\alpha+1)}(-1)^{n+k}z^k.
\end{equation}
Then using the formula
\[
{Q}_{n}(z,\alpha)=
\mathfrak{S}_t\left(\frac{\widetilde{L}_{n}(z,\alpha)-\widetilde{L}_{n}(t,\alpha)}{z-t}\right)\]
one can find the Lanczos polynomials of the second kind:
\begin{equation}\label{ex.l.5}\begin{split}
{Q}_{n}(z,\alpha)&=
\mathfrak{S}_t\left(\sum_{k=0}^{n}\binom{n}{k}\frac{\Gamma(n+\alpha+1)}{\Gamma(k+\alpha+1)}(-1)^{n+k}\frac{z^k-t^k}{z-t}\right)=\\
&=\sum_{k=1}^{n}z^{k-1}\Gamma(\alpha+n+1)\sum_{j=0}^{n-k}(-1)^{n+k+j}\binom{n}{k+j}\frac{\Gamma(\alpha+j+1)}{\Gamma(\alpha+k+j+1)}.
\end{split}\end{equation}
Due to \eqref{2p.new8.r7} the Stieltjes polynomials of the first and second kind are calculated by
  \[\begin{split}
  &P^+_{2n}(z,\alpha)=\Gamma(1+\alpha)\sum_{i=0}^n\binom{n}{i}\frac{(-1)^i z^i}{\Gamma(1+\alpha+i)},\\&
  P_{2n-1}^+(z,\alpha)=-\frac{1}{(n-1)!}\prod_{j=1}^n(j+n)\left(\sum_{i=0}^{n-1}\frac{(-1)^{i-1}(n-1)!i\,z^i}
  {i!(n-i)!\Gamma(1+\alpha+i)}+\frac{(-1)^{n-1}z^n}{\Gamma(1+\alpha+n)}\right),\\&
  Q^+_{2n}(z,\alpha)=\sum_{i=1}^{n}z^{i-1}\Gamma(1+\alpha)\sum_{j=0}^{n-i}(-1)^{i+j+1}\binom{n}{i+j}\frac{\Gamma(1+\alpha+n)}{\Gamma(1+\alpha+i+j+1)},\\&
  Q^+_{2n-1}(z,\alpha)=\Gamma(1+\alpha+n)\sum_{i=1}^{n}z^{i-1}
        \sum_{j=0}^{n-i}\frac{(-1)^{i+j-1}\Gamma(1+\alpha+j)(n-1)!}{i!(n-j-i)!\Gamma(1+\alpha+i+j)}.
  \end{split}
  \]
The set of solutions of the truncated moment problem $MP_\kappa^k(\mathbf{s},2n-1)$ is described by the formula~\eqref{eq:LFT_W}.

Notice that the full moment problem is determinate and  in the case
$\alpha>-1$ its unique solution of $MP_0^0(\mathbf{s})$  is given by
\[
f(z)=\int_0^\infty \frac{x^{\alpha}e^{-x}}{x-z}dx
=z^{\frac{\alpha-1}{2}}e^{-\frac{z}{2}}W_{\frac{-\alpha-1}{2},\frac{\alpha}{2}}(-z),
\]
where  $W_{\alpha,\beta}$ is the Whitteker function, see~\cite[9.222]{GR63}.

If $-2k-1<\alpha<-2k$, $k\in\dN$, then the  solution  $f(z)$ of $MP_k^k(\mathbf{s})$
can be found by the formula
\begin{equation}\label{eq:Weyl_f}
  f(z)=\mathfrak{S}_x\left(\frac{e^{-x}x^\alpha}{x-z}\right)=\int_0^\infty x^{\alpha}\left(h(x,z)-\sum_{j=0}^{2k-1}\frac{h_x^{(j)}(0,z)}{j!}x^j\right)dx,
\end{equation}
where $h(x,z)=\frac{e^{-x}}{x-z}$. Similar formula holds also  for the case $-2k<\alpha<-2k+1$. In particular, if $-2<\alpha<-1$, then
\[
f(z)=I(z)+I'(z)\in MP_1^0(\mathbf{s}),\quad\mbox{ where }
\,I(z)=\frac{z^{\frac{\alpha}{2}}e^{-\frac{z}{2}}}{\alpha+1}W_{\frac{-\alpha-2}{2},\frac{\alpha+1}{2}}(-z).
\]
By Propositions~\ref{prop:Diag_Pade}, \ref{prop:sub_Diag_Pade} the Pad\'{e} approximants of $f$ take the form (see~\eqref{eq:Pade_con2} and~\eqref{eq:Pade_con3}):
\[
f^{[n/n]}(z)=\frac{Q_{2n}^+(z,\alpha)}{P_{2n}^+(z,\alpha)},\quad
f^{[n/n-1]}(z)=\frac{Q_{2n-1}^+(z,\alpha)}{P_{2n-1}^+(z,\alpha)},\quad n\in\dN.
\]


\begin{thebibliography}{1}

\bibitem{Akh}
 {N.\,I. Akhiezer},
\emph{The classical moment problem}, Oliver and Boyd, Edinburgh, 1965.

\bibitem{ADL04}
 {D.~Alpay},  {A.~Dijksma}  and  {H.~Langer},
 {\emph{Factorization of $J$-unitary matrix polynomials on the line and a Schur
  algorithm for generalized \emph{N}evanlinna functions}, Linear Algebra Appl.}
  \textbf{387} (2004), 313--342.


\bibitem{BaGr86}
        G. Baker and P. Graves-Morris, \textit{Pad\'e approximants. Part I:
        Basic Theory}, Encyclopedia of Mathematics and Its Applications,
         Addison-Wesley, London, \textbf{13}, 1981.

\bibitem{BSS00}
 {R.~Beals},  {D.\,H. Sattinger},  and  {J.~Szmigielski},
\emph{Multipeakons and the classical moment problem},
 {Adv. Math.}, \textbf{154} (2000), no. 2, 229--257.

\bibitem{Bus10}%
V. I. Buslaev, \emph{On Hankel determinants of functions given by their expansions in $P-$fractions},
Ukrainian Mathematical Journal, {\bf 62} (2010), no. 3,  358--372.
\bibitem{Der03}
 {M.~Derevyagin},
 \emph{On the Schur algorithm for indefinite moment problem}, Methods Functional
  Anal. Topol., \textbf{9} (2003), no. 2, 133--145 .

\bibitem{DD04}
M.~Derevyagin and V.Derkach,
\emph{ Spectral problems for generalized Jacobi matrices}, Linear
  Algebra Appl.,  \textbf{382} (2004), no. 1, 1--24.

\bibitem{DD07}
M.~Derevyagin and V.Derkach,
\emph{On the convergence of Pad\'{e} approximations for generalized nevanlina function}, Trans. Moscow. Math. Soc., \textbf{68} (2007),
119--162.

\bibitem{D91}
V.~Derkach,
 \emph{Generalized resolvents of a class of Hermitian operators in a
  Kre\u{\i}n space}, Dokl. Akad. Nauk SSSR, \textbf{317} (1991), no. 4, 807--812.

\bibitem{D97}
 {V.~Derkach},
 \emph{On indefinite moment problem and resolvent matrices of Hermitian operators
  in Kre\u{\i}n spaces}, Math.Nachr., \textbf{184} (1997), 135--166.
\bibitem{D98}
 {V.~Derkach},\emph{ Extensions of Laguerre Operators in Indefinite Inner Product Spaces},
 Mathematical Notes, \textbf{63} (1998), no. 4, 449–-459.
  \bibitem{D99}
  V. Derkach, \emph{{On generalized resolvents of Hermitian relations in Krein spaces},J. Math. Sci.} \textbf{97} (1999), no.5, 4420--4460.
\bibitem{DD09}
        V.~Derkach, H.~Dym,
        O\emph{n linear fractional transformations associated with
        generalized $J$-inner matrix functions},
        Integ. Eq. Oper. Th.,  {\bf 65} (2009), no. 1, 1-50.

\bibitem{DHS12}
 {V.~Derkach},  {S.~Hassi},  and  {H.\,S.\,V. de~Snoo},
 {\emph{Truncated moment problems in the class of generalized Nevanlinna
  functions}, Math. Nachr.} \textbf{285} (2012), 1741--1769.

\bibitem{DK15}
 {V. Derkach},  {I. Kovalyov},
 {\emph{On a class of generalized Stieltjes continued fractions}, Methods of Funct.
  Anal. and Topology} \textbf{21} (2015), no. 4, 315--335.

\bibitem{DK16}
V.Derkach, I.Kovalyov, \emph{An operator approach to  indefinite Stieltjes moment problem},
J.Math. Sci., \textbf{227} (2017), 33--67.

\bibitem{DK17}
V.Derkach I.Kovalyov,
\emph{The Schur algorithm for indefinite Stieltjes  moment problem}, Math. Nachr., \textbf{290} (2017), no. 10, 1637–-1662.

\bibitem{DM87}
 {V.~Derkach} and  {M.~Malamud},
 {\emph{On Weyl function and Hermitian operators with gaps}, Doklady Akad. Nauk
  SSSR} \textbf{293} (1987), no. 5, 1041--1046.
\bibitem{DM95}
 {V.~Derkach} and  {M.~Malamud},
 {\emph{The extension theory of Hermitian operators and the moment problem}, J.of
  Math.Sci.} \textbf{73} (1995), no. 2, 141--242.
\bibitem{DM97}
 {V.~Derkach} and  {M.~Malamud},
 \emph{On some classes of holomorphic operator functions with nonnegative
  imaginary part}, 16th OT Conference Proceedings, Operator theory, operator
  algebras and related topics, Theta Found.Bucharest, Timisoara,
  1997, 113-147 .


\bibitem{Dym89}
H.~Dym, \emph{On Hermitian block Hankel matrices, matrix polynomials,
the Hamburger moment problem, interpolation and maximum entropy},
Integral Equations Operator Theory, \textbf{12} (1989), no. 6, 757--812.

\bibitem{EKost14}
 {J.~Eckhardt} and  {A.~Kostenko},
 {\emph{An isospectral problem for global conservative multi-peakon solutions of
  the Camassa-Holm equation}, Comm. Math. Phys.} \textbf{329} (2014), no. 3, 893--918.
 \bibitem{EKost18}
 {J.~Eckhardt} and  {A.~Kostenko},
 \emph{The classical moment problem and generalized indefinite strings}, Integral Equations Operator Theory, {\bf 90} (2018), no. 2, 1--30.

\bibitem{Fuhr10}
P.A. Fuhrmann,  \emph{A polynomial approach to linear algebra}, Second edition,Universitext. Springer, New York, 2012.


\bibitem{GG84}
V.~I.~Gorbachuk and M.~L.~Gorbachuk, \emph{Boundary problems for
differential operator equations}, Naukova Dumka, Kiev, 1984
(Russian).

\bibitem{GR63}
    I.S. Gradshtein, I.M. Ryzhik,\emph{ Tables of integrals, sums and products}, Fizmatgiz, Moscow, 1963.

\bibitem{KacK68}
    I.S.~Kac, M.G.~Kre\u{\i}n, \emph{$R$-functions -- analytic functions
    mapping the upper halfplane into itself}, Supplement to the Russian
    edition of F.V.~Atkinson, \textit{Discrete and continuous boundary
    problems}, Mir, Moscow 1968 (Russian) (English translation: Amer.
    Math. Soc. Transl. Ser. 2, {\bf 103} (1974), 1--18).


\bibitem{K17}I.Kovalyov, \emph{A truncated indefinite Stieltjes moment problem}, J.Math.Sci., \textbf{224} (2017), 509--529.




\bibitem{Kr52}
 {M.\,G. Kre\u{\i}n},
 {On a generalization of investigations of Stieltjes (Russian), Doklady
  Akad. Nauk SSSR (N.S.)} \textbf{87}, 881-884 (1952).



\bibitem{KL77}
M.~G. Kre\u{\i}n, H.~Langer,
U\emph{ber einige Fortsetzungsprobleme, die eng mit der Theorie
  Hermitscher Operatoren in Raume $\Pi_{\kappa}$ zusammenh\"{a}ngen}, I, Einige
  Fuktionenklassen und ihre Dahrstellungen, Math. Nachr., \textbf{77} (1977), 187--236.
\bibitem{KL79}
M.~G. Kre\u{\i}n and H.~Langer.
 \emph{On some extension problem which are closely connected with the
  theory of Hermitian operators in a space $\Pi_{\kappa}$ III, Indefinite
  analogues of the Hamburger and Stieltjes moment problems, Part I,}
  Beitr{\"a}ge zur Anal., \textbf{14} (1979), 25--40.

\bibitem{KL81}
M.~G. Kre\u{\i}n and H.~Langer.
\emph{ On some extension problem which are closely connected with the
  theory of Hermitian operators in a space $\Pi_{\kappa}$ III. Indefinite
  analogues of the Hamburger and Stieltjes moment problems, Part II,}
  Beitr{\"a}ge zur Anal., \textbf{15} (1981), 27--45.


\bibitem{Kro81}
L. Kronecker,
\emph{Zur Theorie der Elimination einer variabeln aus zwei algebraischen Gleichungen}, Monatsberichte, (1881), 535--600.


\bibitem{L86}
        H. Langer, \emph{A characterization of generalized zeros of
        negative type of functions of the class
        $\mathbf{N}_{\kappa}$}, Oper. Theory Adv. Appl., {\bf 17} (1986),
         201--212.
\bibitem{LSh92}
        H. Langer, A. Schneider, \emph{Some remarks about polynomials which are orthogonal
with respect to an indefinite weight}, Results in Mathematics
{\bf 21} (1992), 152--164.
\bibitem{Mag62}
 {A.~Magnus},
 {\emph{Expansion of power series into $P$-fractions}, Math. Zeitschr.},
  \textbf{80} (1962), 209--216.
\bibitem{MK78}
R.D. Morton, A.M. Krall,
\emph{Distributional weight functions for orthogonal polynomials},
SIAM J. Math. Anal. {\bf 9} (1978), no. 4, 604--626.
\bibitem{Sim98}
B. Simon, \emph{The classical moment problem as a self-adjoint finite difference operator}, Adv. Math.,
{\bf 137} (1998), 82--203.
\bibitem{seg} G. Szeg\"{o},  \emph{ Orthogonal Polynomials}, Fourth edition, AMS, Providence, RI, 1975.

\bibitem{St94}{  T. Stieltjes,}
  \emph{Recherches sur les fractions continues},
  Ann.de Toulouse, {\bf 8} (1894) 1-122.
\bibitem{suetin} P. K. Suetin, \emph{Classical orthogonal polynomials},  2nd rev. ed., Nauka, Moscow,
1979.

\bibitem{Wall}
 {H.\,S. Wall},
\emph{Analytic theory of continued fractions}, Chelsey, New York, 1967.

\end{thebibliography}
\end{document}